\documentclass[10pt]{article}

\usepackage{latexsym}
\usepackage{amsmath}
\usepackage{amssymb}
\usepackage{amsthm}

\newtheorem{theorem}{Theorem}[section]
\newtheorem{proposition}{Proposition}[section]
\newtheorem{lemma}{Lemma}[section]
\newtheorem{corollary}{Corollary}[section]

\newtheorem{remark}{Remark}[section]

\numberwithin{equation}{section}

\begin{document}

\title{Schwarz Symmetrization and Comparison \\ Results for Nonlinear Elliptic Equations \\ and Eigenvalue Problems}

\author{L. P. BONORINO$^1$ AND J. F. B. MONTENEGRO$^2$ \\ \\ {\normalsize
 $^1$Departamento de Matem\'atica Pura e
Aplicada \hfill } \\ {\normalsize Universidade Federal do Rio Grande do Sul \hfill } \\
{\normalsize 91509-900, Porto Alegre, RS, Brazil \hfill } \\
{\normalsize $^2$Departamento de Matem\'atica, Universidade
Federal do Cear\'a \hfill} \\ {\normalsize 60455-760, Fortaleza,
CE, Brazil \hfill} \\ }

\date{}

\maketitle

\begin{abstract}
We compare the distribution function and the maximum of solutions
of nonlinear elliptic equations defined in general domains with
solutions of similar problems defined in a ball using Schwarz
symmetrization. As an application, we prove the existence and
bound of solutions for some nonlinear equation. Moreover, for some
nonlinear problems, we show that if the first $p$-eigenvalue of a
domain is big, the supremum of a solution related to this domain
is close to zero. For that we obtain $L^{\infty}$ estimates for
solutions of nonlinear and eigenvalue problems in terms of other
$L^p$ norms.
\end{abstract}



\section{Introduction}
In this work we study the $L^{p}-$norm and the distribution
function of solutions to the Dirichlet Problem
\begin{equation*}
\left\{
\begin{array}{rcll}
- {\rm div} ( a(u, \nabla u)) & = & f(u) & \; {\rm in } \; \; \Omega \\[2pt]
  u & = & 0 & \; {\rm on } \; \; \partial \Omega, \\ \end{array} \right.
\end{equation*}
where $\Omega$ is an open bounded set in $\mathbb{R}^n$,
$f:\mathbb{R} \to \mathbb{R}$ and $a:\Omega \times \mathbb{R} \times \mathbb{R}^n \to
\mathbb{R}^n$ satisfy some suitable conditions. First we assume the following hypotheses:

\

\noindent (H1) $f$ is a nonnegative locally Lipschitz function;
\\(H2) $f$ is nondecreasing;
\\(H3) $a\in C^0(\mathbb{R} \times \mathbb{R}^n; \mathbb{R}^n ) \cap C^1(\mathbb{R} \times (\mathbb{R}^n \backslash \{0\}); \mathbb{R}^n)$ is given by $a(t,{\rm z}) =  e(t,|{\rm z}|){\rm z}$, where $e \in C^1(\mathbb{R} \times (\mathbb{R} \backslash \{0\}))$ is  positive on $\mathbb{R} \times \mathbb{R} \backslash \{0\}$,
$a(t,0)=0$, $a(t,{\rm z})\cdot {\rm z}$ is convex in the variable ${\rm z} \in \mathbb{R}^n$ and $\partial_{s}\left( |a(t,s {\rm z})|\right) > 0$ for ${\rm z} \ne 0$ and $s > 0$. Observe that the convexity of ${\rm z} \mapsto a(t,{\rm z}) \cdot {\rm z}$ implies in the fact that $s \mapsto |a(t,s{\rm z})|$ is increasing.
\\(H4) there exist $p\ge q >1$, $q_0 > 1$, and positive constants $C_s$, $C_{*}$ and $C^{*}$ s.t.
$$ C_{s} |{\rm z}|^{q_0} \le  \langle a(t, {\rm z}) , {\rm z}  \rangle  \quad {\rm for } \quad |{\rm z}| \le 1, \; t \in \mathbb{R}$$  and $$ C_{*} |{\rm z}|^q \le \langle a(t, {\rm z}) , {\rm z}  \rangle  \le C^{*}(|{\rm z}|^p + |t|^p + 1)  \quad {\rm for } \quad |{\rm z}| \ge 1, \; t \in \mathbb{R}.$$
Hence, using that $s \mapsto a(t,s{\rm z})\cdot s{\rm z}$ is increasing and positive,
$$ C_*(|{\rm z}|^q - 1 ) \le a(t,{\rm z}) \cdot {\rm z} \le C^*(|{\rm z}|^p + |t|^p + 1) \quad {\rm for }  \quad {\rm z} \in \mathbb{R}^n $$
and
\begin{equation}
 \label{alqestimate}
C_*( \lambda_B \|w\|_q^q -|\Omega|) \le \int_{\Omega} a(w, \nabla w) \cdot \nabla w \, dx \le  C^*( \|\nabla w \|_p^p  + \| w \|_p^p + |\Omega|) ,
\end{equation}
for $w \in W^{1,p}_0(\Omega)$, where $\lambda_B$ is the first eigenvalue of $-\Delta_q$ in a ball $B$, that has the same measure as $\Omega$.
\\(H5) there exist $\beta\ge 0$ and $\alpha <  C_* \lambda_B $ such that
$$ 0< f(t) \le \alpha t^{q-1}+\beta \quad {\rm for} \quad t >0.
$$

\

At first our main concern is to compare the maximum and the
distribution function of a solution associated to $\Omega$ with
one associated to $B$. We can obtain even a priori estimates of
solutions for some problems with nonlinear lower order terms and
prove the existence of solution. Later on we see also some
applications for these estimates, including $L^{\infty}$ estimates
for some eigenvalue and nonlinear problems. So we show that if a
domain is ``far away'' from the ball (ie, its first $p$-eigenvalue
is big), then the maximum of a solution is small. Indeed the
supremum of a solution is bounded by some negative power of the
first $p$-eigenvalue. This kind of question seems to be new and
the works in the literature normally are focused in comparing
solutions with a radial one, disregarding better estimates when
the domain is not close to a ball.

\

More precisely, let $B$ be the open ball in $\mathbb{R}^n$,
centered at the origin,  such that $|B|= |\Omega|$, where $|C|$
denotes the Lebesgue's measure in $\mathbb{R}^n$ of a measurable
set $C$, and consider the function $U_B$ given by
\begin{equation}
U_B(x) = \sup \{ U(x) \; | \; U \in W^{1,p}_0(B) \; \text{is a
radial solution of } \; ({\rm \tilde{P}}_B) \}, \label{definicaoU}
\end{equation} where $({\rm \tilde{P}}_B)$ is the
Dirichlet Problem
\begin{equation*} \left\{
\begin{array}{rcll}
- {\rm div} ( {\tilde a}(U, \nabla U)) & = & f(U) & \; {\rm in } \; \; B \\[5pt]
  U & = & 0 & \; {\rm on } \; \; \partial B. \\ \end{array}
\right.  \tag{${\rm \tilde{P}}_{B}$}
\end{equation*}
Let $u$ be a weak solution of
\begin{equation*} \left\{
\begin{array}{rcll}
- {\rm div} ( a(v, \nabla v)) & = & f(v) & \; {\rm in } \; \; \Omega \\[5pt]
  v & = & 0 & \; {\rm on } \; \; \partial \Omega. \\ \end{array}
\right.  \tag{${\rm P}_{\Omega}$} \label{eq1}
\end{equation*}
in $W_0^{1,p}(\Omega)$. Observe that $u$ and $U_B$ are
positive. Define the distribution function of $u$ by
$$\mu_{u}(t)=|\{ x \in \Omega : u(x) > t \}|.$$
For $a$, $\tilde{a}$ and $f$ satisfying hypotheses (H1)-(H5) (the
constants and powers related to $a$ and $\tilde{a}$ can be
different) and $\tilde{a}(t,{\rm z})\cdot {\rm z}\le a(t,{\rm
z})\cdot {\rm z}$, we prove that $U_B$ is a solution of $({\rm
\tilde{P}}_B)$ and, in Theorem \ref{finaltheorem},
\begin{equation}
\mu_{u}(t) \le \; \mu_{B}(t), \quad \forall t \in [0,\max U_B],
\label{comparacao}
\end{equation}
where $\mu_{B}$ is the distribution function of $U_B$. If $\Omega$
is not a ball, $a=a({\rm z})$ and $(a({\rm z})\cdot {\rm
z})^{1/p}$ is convex, then this inequality is strict.

We also prove some sort of maximum principle with respect to the
solutions in the ball in the following sense: if $u$ and $U$ are
solutions of $({\rm P}_{\Omega})$ and $({\rm \tilde{P}}_B)$
respectively,
 $u^{\sharp} \le U$ (not necessarily maximal solution) and $u^{\sharp} \ne U$, then $u^{\sharp} < U$ provided $f$ and $a$ satisfy suitable conditions.

\

As an application we obtain this comparison to the problem with
lower order terms
\begin{equation}
 \left\{ \begin{array}{rclc}
 \displaystyle -{\rm div} ({a(\nabla u)}) - \frac{h'(u)}{h(u)} \nabla u \cdot a(\nabla u) & = & g(u) & {\rm in} \; \; \Omega \\[3pt]
    u & = & 0 & { \rm on } \; \; \partial \Omega,
\end{array} \right.
\label{appequation}
\end{equation}
where $h \in C^1$ is bigger than some positive constant, $f=gh$ and $a_1(t,{\rm z})=h(t)a({\rm z})$ satisfy (H1)-(H5). This holds even if $h$ has a bad growth and does not satisfy the upper inequality of \eqref{alqestimate}. For the special case $$ -\Delta_p u - \frac{h'(u)}{h(u)}|\nabla u|^p = g(u)$$
this priori estimate can be used to prove existence of solution.

\

\noindent Moreover we get also some result even when $f$ is not
nondecreasing. Indeed, if $f$ is positive, $f(t)/t^{p-1}$ is
decreasing and $a(t,{\rm z}) =\tilde{a}(t,{\rm z})= {\rm z}|{\rm
z}|^{p-1}$, we show that
$$\max U_B \ge \max u.$$

 \noindent This $L^{\infty}$ estimate can be easily extended to the problem
\begin{equation}
\left\{
\begin{array}{rclc}
- \Delta_p v + k(v) & = & f(v) & {\rm in } \; \; \Omega \\[3pt]
  v & = & 0 & {\rm on } \; \; \partial \Omega ,\\ \end{array}
  \right.
\label{extensioncaparison}
\end{equation}
where $k$ is positive and nondecreasing and $f$ is positive and
$f(t)/t^{p-1}$ is decreasing.

 \

\noindent Then we apply these results to prove that if $w \in W_0^{1,p}(\Omega)$ is a solution of ${\rm div}(a (x, \nabla w)) = f(w)$ in  $\Omega$,
where $a$ satisfies some conditions and $f\in C^1(\mathbb{R})$ is bounded by
$c|t|^{q-1} + d$, with $1 < q \le p$ and $c,d \ge 0$, then
$$ \|w\|_{\infty} \le C_1 \|w\|_r^{\frac{rp}{n(p-q)+rp}} + C_2 \|w\|_{r}^{\frac{rp}{n(p-1) + rp}},  $$
where $C_1=C_1(n,p,q,r,\rho,c)$ and $C_2=C_2(n,p,r,\rho,d)$ are positive constants.
In the special case $|\Delta_p w| \le |\lambda|  |w|^{q-1}$, where $\lambda \in \mathbb{R}$, we have
\begin{equation} \label{eigenvalueestimate00} \|w\|_s  \le \left[ \frac{2}{(\omega_n)^{1/r}} \left( \frac{2(p-1)}{p}\right)^{\frac{n(p-1)}{rp}} \left(\frac{|\lambda|}{n}\right)^{n/rp}\right]^{\frac{s-r}{\kappa s}} \| w\|_r^{\frac{s-r}{\kappa s}+ \frac{r}{s}} ,
 \end{equation}
where $0< r < s$ and $\kappa = 1+\frac{n(p-q)}{rp}$. These
inequalities imply, according to Corollary
\ref{Linfinitestimatebyeigenvalue}, in a $L^{\infty}$-norm decay
of the solutions of some sublinear equations, when the domain
becomes ``far away'' from a ball with the same volume. Since the
ball is the domain of a given measure that maximizes the $L^p$
norms in several problems, it would be interesting to obtain
better estimates for solutions that are not defined in a ball.
Hence, we need to measure in some way the difference between its
domain and the corresponding ball. The first eigenvalue is a
possible form of distinction between these sets, that we use to
establish some upper bound. Finally, as an application, we prove
that $u^{\sharp} < U$, where $u$ is a solution of $({\rm
P}_{\Omega})$ and  $U$ a solution of $({\rm \tilde{P}}_B)$, even
when $f$ is not monotone, provided the first eigenvalue associated
to $\Omega$, $\lambda_p(\Omega)$, is big enough and some
conditions on $a$ and $f$ are satisfied.

We point out that we are not interested in establishing existence of solutions for $({\rm P}_{\Omega})$. Our main concern is just to compare these solutions and we obtain existence results only for the radial case.

\

Results of this type have been obtained by several authors. In
\cite{T1} Talenti proved that if $u$ is the weak solution of the
Dirichlet Problem
$$ - \sum_{i,j=1}^n \frac{\partial}{\partial x_i} \left( a_{ij}(x)
\frac{\partial u}{\partial x_j} \right) + c(x)u = f(x) \; \; {\rm
in } \; \;  \Omega \quad {\rm and} \quad u =0 \; \; {\rm on} \; \;
\partial \Omega,
$$ where $c(x) \ge 0$, $\sum_{ij}a_{ij}(x)\xi_i \xi_j \ge
\xi_1^2 + \dots +\xi_n^2$ and $v$ is the weak solution of
$$-\Delta v = f^{\sharp} \; \; {\rm in} \; \; B \quad {\rm and } \quad v =0 \; \; {\rm on} \; \; \partial B,$$
where $B$ is the ball centered at 0 such that $|B| = |\Omega|$ and
$f^{\sharp}$ is the decreasing spherical rearrangement of $f$,
then ${\rm ess \, sup} u \le {\rm ess \, sup} v$ and $\mu_u \le
\mu_v$. As a consequence, $\| v \|_{L^p} / \| f^{\sharp}\|_{L^q}
\ge \| u \|_{L^p} / \| f\|_{L^q}$. This estimate is an extension
of the one previously obtained by Weinberger \cite{W} for the
ratio $\| u \|_{L^{\infty}} / \| f\|_{L^q}$. Further results have
been proved for a larger class of linear equations that either
satisfy weaker ellipticity conditions (see \cite{AT1}, \cite{AT2})
or contain lower order terms (see \cite{ALMT}, \cite{ALT2},
\cite{AMT1}, \cite{AT3}, \cite{C}, \cite{FP}, \cite{T3},
\cite{TV}). Similar problems were studied in \cite{Ke1}, \cite{Ke2}, \cite{Ke3}.

\

As in the linear case, estimates have been obtained for solutions $u \in W_0^{1,p}(\Omega)$ to the nonlinear
problem
$$  \begin{array}{l} - \sum_{i=1}^n ( a_i(x,u, \nabla u))_{x_i} - \sum_{i=1}^n (b_i(x)|u|^{p-2}u)_{x_i} + h(x,u)= f(x,u) \; {\rm in } \;
\Omega , \end{array}  $$
comparing the decreasing spherical rearrangement of $u$ with the
solution of some nonlinear ``symmetrized" problem. For instance,
the case $b_i=h=0$ and $\sum a_i(x,u,\xi)\xi_i \ge A(|\xi|)$,
where $A$ is convex and $\lim_{r \to 0} A(r)/r=0$, is considered
in \cite{T2}. The problem in a general form is studied in
\cite{BM}, assuming that the coefficients are in suitable spaces
and $\sum a_i(x,u,\xi)\xi_i \ge |\xi|^p$. Under similar
hypotheses, the case $b_i=0$ is considered in \cite{FM} and
different comparison results are obtained. In \cite{AFLT}
estimates are proved when the coefficients satisfy $b_i=h=0$,
$a_i=a_i(Du)$ and $\sum a_i(\xi)\xi_i \ge (H(\xi))^2$, where $H$
is a nonnegative convex function, positively homogeneous of degree
1. Other related result were established in \cite{Ab}, \cite{FM2}, \cite{Me}.
Some results also extend to parabolic equations (see e.g.
\cite{AFLT}, \cite{ALT2}, \cite{B2}).

\

Usually comparison results are obtained considering a
``symmetrized equation" that is different from the original one.
In this work we can keep the original equation and symmetrize only
the domain, obtaining sharper estimates. Results similar to ours
are established in  \cite{B1}, \cite{L1} for the laplacian
operator, where the authors apply the method of subsolution and
supersolution to prove that, for a given symmetric solution $U$ in
the ball, there exists some solution in $\Omega$ for which the
symmetrization is less than $U$. Indeed, applying the iteration
procedure used in those works and the main result of \cite{T2},
the estimate \eqref{comparacao} can be obtained in the particular
case $-{\rm div}(a(\nabla u)) = f(u)$, provided we have some a
priory estimate in the $L_q$ norm for subsolutions and the
existence of the maximal radial solution $U_B$. Using different
techniques, we prove in Section \ref{MR} that the symmetrization
of any solution of (\ref{eq1}) is bounded by $U_B$, even in the
case $a=a(t,z)$ and $\tilde{a}=\tilde{a}(t,z)$, as long as
hypotheses (H1)-(H5) are satisfied. In Section \ref{PR}, we review
some important concepts and results. Some estimates in this
section are interesting by itself. In Section \ref{SC} we get
estimates assuming that $a({\rm z})=\tilde{a}({\rm z})=|{\rm
z}|^{p-2}{\rm z}$ and $f(t)/t^{p-1}$ is decreasing. Indeed we
prove that $\max U_B \ge \max u$ even when $f$ is not
nondecreasing. Observe that the uniqueness of solution to the
problems (\ref{eq1}) and $({\rm \tilde{P}}_B)$ is proved in
\cite{BO} for the Laplacian operator when $f(t)/t$ is decreasing.
An extension of this is proved to the $p$-Laplacian in \cite{BK}.
Hence, some results in this section can be obtained directly from
the existence of a solution associated to $B$ that is greater than
some solution associated to $\Omega$. In Section \ref{RC}, we
study the behavior of solutions in the radial case. In Section
\ref{EBR} we obtain a bound to solutions of \eqref{appequation}
and in some special case we use this comparison to show the
existence of solution. In Section \ref{EP} we get some
inequalities between the $L^p$ norms of solutions of some
``eigenvalue problems'' and some lower bound for the distribution
function of these solutions. For eigenvalue problems, the $L^p$
estimates are established in \cite{AFT}, \cite{C2}, and \cite{C3},
where the authors obtain sharper estimates, since the constants
are optimal. We are not concerned with the best constant but only
with the relations between the $L^p$ norms and the real parameter
$\lambda$. We get an explicit relation for a larger class of
equations and, for the typical eigenvalue problem, the estimate
hold not only for the first eigenvalue of the operator but also
for the others. Other authors make some similar estimates on
manifolds (see e.g. \cite{G2} and \cite{H2}) for the classical
eigenvalue problem, but the constant depends on the manifold and
the boundary. It is also established some $L^p$ estimates for a
class of Dirichlet problems and a relation between the norms and
the first eigenvalue of the domain.

\section{Preliminary Results}
\label{PR}

In this section we recall some important definitions and useful
results. First, if $\Omega$ is an open bounded set in
$\mathbb{R}^n$ and $u: \Omega \to \mathbb{R}$ is a measurable
function, the distribution function of $u$ is given by
$$\mu_{u}(t)=|\{ x \in \Omega : |u(x)| > t \}| \quad {\rm for } \; t \ge 0.$$
The function $\mu_u$ is non-increasing and right-continuous. The
decreasing rearrangement of $u$, also called the {\it generalized
inverse} of $\mu_u$, is defined by
$$u^{*}(s)= \sup \{ t \ge 0 : \mu_u (t) \ge s \}. $$
If $\Omega^{\sharp}$ is the open ball in $\mathbb{R}^n$, centered at
$0$, with the same measure as $\Omega$ and $\omega_n$ is the
measure of the unit ball in $\mathbb{R}^n$, the function
$$ u^{\sharp}(x) = u^*( \omega_n |x|^n) \quad {\rm for} \quad x \in \Omega^{\sharp} $$
is the spherically symmetric decreasing rearrangement of $u$. It
is also called the Schwarz symmetrization of $u$. For an
exhaustive treatment of rearrangements we refer to \cite{ALT1},
\cite{B1}, \cite{CR}, \cite{HLP}, \cite{K2}, \cite{M}. The next
remark reviews important properties of rearrangements and will be
necessary through this work.
\begin{remark}
\label{obs2} Let $v,w$ be integrable functions in $\Omega$ and let
$g:\mathbb{R}\to\mathbb{R}$ be a non-decreasing nonnegative function.
Then
$$ \int_{\Omega} g(|v(x)|) \; dx = \int_0^{|\Omega|} g(v^*(s)) \; ds
= \int_{\Omega^{\sharp}} g(v^{\sharp}(x)) \; dx .$$ Hence, if
$\mu_v(t) \ge \mu_w(t) $ for all $t>t_1>0$, it follows that
$$ \int_{t_1<v} \! \! g(v(x)) \; dx \! =  \! \int_0^{\mu_u(t_1)} \! \! g( v^*(s)) \; ds \!
\ge \! \int_0^{\mu_w(t_1)} \! \! g( w^*(s)) \; ds \! = \!
\int_{t_1<w} \! \! g(w(x)) \; dx,
$$ since $v^*(s) \ge w^*(s)$
for $s \le \mu_w(t_1) $. Moreover, if $|\{v > t_2\}| \le |\{w >
t_2\}|<\infty$, $|\{v> t_1\}|=|\{w>t_1\}|<\infty$ and $|\{v> t\}|
\ge |\{w>t\}|$ for all $t_1<t<t_2$, then
$$ \int_{ t_1 < v \le t_2} g(v(x)) \; dx \; \ge \; \int_{t_1 < w \le t_2 }g(w(x)) \;dx. $$
Finally, an extension of the P\'olya-Szeg\"o principle (\cite{HjSq}, see also
\cite{Br}, \cite{BZ}, \cite{K2}) states that, if $B = B(t,{\rm z}) \in C([0,\infty) \times [0,\infty))$ is increasing and convex in the variable ${\rm z}$, then
$$\int_{\Omega} B( v(x), \nabla v(x))  \; dx  \; \ge \; \int_{\Omega^{\sharp}} B( v^{\sharp} (x), \nabla v^{\sharp} (x)) \;
dx \quad {\rm for} \; \; v\ge 0 \; \; {\rm in} \; W_0^{1,p} (\Omega). $$ This inequality
also holds if we replace $\Omega$ and $\Omega^{\sharp}$ by $\{t_1
< v < t_2 \}$ and $\{t_1 < v^{\sharp} < t_2 \}$, respectively. For
$B(t,{\rm z})=|{\rm z}|^2$, this inequality is the classical version of the
P\'olya-Szeg\"o \cite{PS} principle.
\end{remark}

\begin{remark} \label{classicalsol} For any bounded open set $\Omega'$ satisfying $|\Omega'| \le |\Omega|$, there exists a constant
$C=C(n,q,\alpha,\beta,C_{*},|\Omega|,|\Omega'|)$ such that
$\sup u \le C$ for any weak solution $u \in W_0^{1,p}(\Omega')$ of $({\rm P}_{\Omega'})$. Moreover,
$C=O(|\Omega'|^{\rho})$ as $|\Omega'| \to 0$, where $\rho >0$ depends only on $n$ and $p$.
This result is a consequence of the following two lemmas.
\end{remark}
\begin{lemma} Let $\Omega'$ be a bounded open set s.t. $|\Omega'| \le |\Omega|$.
If $u \in W_0^{1,p}(\Omega')$ is a nonnegative subsolution of $({\rm
P}_{\Omega'})$ and conditions (H1),(H5), $ C_{*} (|{\rm z}|^q -1) \le \langle a(t, {\rm z}) , {\rm z}  \rangle$ for ${\rm z} \in \mathbb{R}^n$, $t \in \mathbb{R}$ are satisfied, then
$$ \|u\|_{L^q} \le  M(\Omega'):= \left( \frac{2 C_* |\Omega'|}{C_* \lambda_{B'} - \alpha}\right)^{1/q} + \left( \frac{ 2 \beta |
\Omega'|^{1/q'}}{C_* \lambda_{B'} - \alpha}\right)^{1/(q-1)},$$
where $1/q' + 1/q =1$, $B'$ is a ball that satisfies
$|B'|=|\Omega'|$ and $\lambda_{B'}$ is the first eigenvalue of
$-\Delta_q$ in $B'$.
\end{lemma}
\begin{proof} Multiplying the equation by $u$ and integrating, we get $$
\int_{\Omega'} \nabla u \cdot a(u,\nabla u) \, dx \le \int_{\Omega'} u
f(u) \, dx  \le \alpha \|u\|_{q}^q + \beta \| u \|_{q} |
\Omega'|^{1/q'}.
$$ Since $ C_{*} (|{\rm z}|^q -1) \le \langle a(t, {\rm z}) , {\rm z}  \rangle$, the first inequality of \eqref{alqestimate} holds. Hence
$$ \| u \|_{q} \left[ (C_* \lambda_{B'} - \alpha) \|u\|_q^{q-1}  - \beta  |\Omega'|^{1/q'}\right] \le  C_* |\Omega'|.
$$
Studying the cases $(C_* \lambda_{B'} - \alpha) \|u\|_q^{q-1}  - \beta  |\Omega'|^{1/q'} \le  (C_* \lambda_{B'} - \alpha) \|u\|_q^{q-1}/2$  and $ >  (C_* \lambda_{B'} - \alpha) \|u\|_q^{q-1}/2$
individually, we get the result.  \end{proof}

\noindent Next lemma is a particular result of Theorem 3.11 of
\cite{MZ} in the case $n \ge q$. For $n < q$, the estimate can be
obtained following the computations of that theorem and Morrey's
inequality. A sketch of the proof is done in the appendix.
\begin{lemma}
Suppose that $u$ satisfies the hypotheses of the preceding lemma.
If $n < q$  then
$$\sup_{\Omega'} u \le C \|u\|_q + D |\Omega'|^{1/q} ,$$
where $C= C(n,q,\alpha, \beta,C_{*})$ and $D= D(n,q,\alpha,
\beta,C_{*})$.

\noindent If $n \ge q$, then
$$ \sup_{\Omega'} u \le C (|\Omega'|^{1/n} + 1)^{\rho}  \left( \frac{ \|u\|_q}{| \Omega' |^{1/q}} + |\Omega'|^{1/n} \right),
 $$ where $\rho=n/q$ and $C= C(n,q,\alpha, \beta,C_{*})$ if $n > q$, and $\rho= \frac{\tilde{q}}{2\tilde{q} - n}$, $\tilde{q}\in(n/2,n)$, and  $C= C(n,\alpha, \beta,C_{*},\tilde{q})$ if $n=q$.
\label{moserLplemma}
\end{lemma}

\noindent From these two lemmas we get, for $n < q$, that
\begin{equation} \label{sup1}
 \sup_{\Omega'} u \le C M(\Omega') + D |\Omega'|^{1/q},
\end{equation} where $C=C(n,q,\alpha,\beta,C_{*})$ and
$D=D(n,q,\alpha,\beta,C_{*})$. For $n \ge q$, it follows that
\begin{equation} \label{sup2} \sup_{\Omega'} u \le C (|\Omega'|^{1/n} + 1)^{\rho}  \left( \frac{ M(\Omega')}{| \Omega' |^{1/q}} + |\Omega'|^{1/n} \right), \end{equation}  where
$C=C(n,q,\alpha,\beta,C_{*})$ if $n > q$ and
$C=C(n,\alpha,\beta,C_{*},\tilde{q})$ if $n=q$. Since
$\lambda_{B'} = \lambda_{B_1} / |B'|^{q/n}$, where $B_1$ is
the unit ball, we have
$$M(\Omega') \le  E |\Omega'|^{\frac{1}{q}+\frac{1}{n}} \quad {\rm if } \; |\Omega'| \le |\Omega|, $$
where $E$ is a constant that depends only on $n,q,\alpha,\beta,C_{*}$ and $|\Omega|$. Using this and inequalities \eqref{sup1} and
\eqref{sup2}, we obtain
\begin{equation} \label{sup3}
\sup u \le C |\Omega'|^{1/q} \quad {\rm for} \; n < q \quad {\rm
and} \quad \sup u \le C |\Omega'|^{1/n} \quad {\rm for}  \; n
\ge q, \end{equation} where $C$ depends only on $n$, $q$, $\alpha$, $\beta$, $C_{*}$, and $\Omega$.
Hence, if $(\Omega_n)$ is a sequence of domains such that
$|\Omega_n|\to 0$ and $(u_n)$ a sequence of solutions of
$(P_{\Omega_n})$, then $\sup |u_n| \le C|\Omega_n|^{\sigma} \to 0$, where $\sigma = 1/q$ or $\sigma=1/n$.

\

\noindent Now we recall some well-known results that appear in many forms.

\begin{lemma} Let $u$ be a weak solution of {\rm (\ref{eq1})} in $W_0^{1,p}$.
Then
$$ \int_{\Omega_t} -u f(u) +  \nabla u \cdot a( u, \nabla u ) \; dx \; = \;
- t \int_{\Omega_t} f(u) \; dx \quad \forall ~t \ge 0 \, ,$$
where $\Omega_t = \{ x \in \Omega \, : \, u(x) > t \} $.
\label{divergencia}
\end{lemma}
\begin{proof}
Let $\psi:\mathbb{R} \to \mathbb{R}$ be the function defined by
$\psi(s) = (s-t) \chi_{\{s > t\}}(s)$. Consider $\varphi: \Omega
\to \mathbb{R}$ given by $\varphi(x)=\psi(u(x))$. Since $\psi$ is
a Lipschitz function and $t > 0$, $\varphi \in W_0^{1,p}$. Furthermore,
$$ \varphi = (u-t)\chi_{\{u > t\}} \quad {\rm and} \quad
\nabla \varphi = \chi_{\{u > t\}}\nabla u \,.$$ Then, since $u$ is a
weak solution of {\rm (\ref{eq1})},
$$ \int_{\Omega} \chi_{\{u > t\}} \nabla u \cdot a( u, \nabla u) \; dx \; = \;
\int_{\Omega} f(u) (u-t)\chi_{\{u > t\}} \; dx \, , $$ proving the
lemma.   \end{proof}

\begin{lemma} Assuming the same hypotheses as in the last lemma,
$$ \int_{\{u=t\}} \frac{\nabla u \cdot a( u,\nabla u)}{ |\nabla u | }\; dH^{n-1} =  \int_{\Omega_t} f(u) \; dx$$
for almost every $t \ge 0$. If $u$ satisfies $u=c$ on $\partial \Omega$, $c\in \mathbb{R}$, then this identity holds for almost every $t \ge c$. \label{divergencelemma}
\end{lemma}
\begin{proof}
For $t_1 < t_2$, from Lemma \ref{divergencia}, we get
\begin{align*}
\int_{A_{t_1t_2}} \! \! - u f(u) + \nabla u \cdot a( u, \nabla u)
\, dx &= t_2 \! \int_{\Omega_{t_2}} \! f(u) \, dx - t_1
\int_{\Omega_{t_1}} f(u) \,
dx \\[5pt]
  &=  (t_2 - t_1) \! \int_{\Omega_{t_2}} \! \! \! f(u) \,  dx - t_1 \! \int_{A_{t_1t_2}} \! \! \! f(u) \, dx,
\end{align*}
where $A_{t_1t_2} = \{ t_1 < u \le t_2\}$. Then,
\begin{equation}
\int_{A_{t_1t_2}}  \nabla u \cdot a \; dx  = (t_2 - t_1)
\int_{\Omega_{t_2}}  f(u) \; dx \; +  \int_{A_{t_1t_2}}  (u-t_1)
f(u) \; dx. \label{estrategica1}
\end{equation}
Hence, using the coarea formula, we obtain
$$\displaystyle \int_{t_1}^{t_2} \int_{\{u=t\}} \! \frac{(\nabla u \cdot a) |\nabla u |^{-1}}{t_2-t_1}  dH^{n-1} dt =  \int_{\Omega_{t_2}}
\! f(u) \; dx \; + \frac{\displaystyle \int_{A_{t_1t_2}} \!
(u-t_1) f(u) \; dx}{t_2-t_1}.$$ Making $t_2 \to t_1$, the integral
in the left hand side converges to the integrand for almost every
$t_1$ and the integral over $\Omega_{t_2}$ converges to a integral
over $\Omega_{t_1}$. The last integral goes to zero, since
$$ \left|  \int_{A_{t_1t_2}}  \frac{(u-t_1)}{t_2-t_1} f(u) \;
dx \right| < \left|  \int_{A_{t_1t_2}} f(u) \; dx \right| \le
f(t_2) |A_{t_1t_2}| \to 0,$$ completing the proof. For the case $u
= c$ on $\partial \Omega$, note that $u-c \in W^{1,p}_0(\Omega)$
is a weak solution of $-{\rm div} \; \bar{a}(v,\nabla v) =
\tilde{f}(v)$, where $\bar{a}(t,{\rm z})= a(t+c,{\rm z})$ and
$\tilde{f}(t)=f(t+c)$. Then, from the previous case, we get
result.
\end{proof}

The following statement is a direct consequence of Brothers
and Ziemer's result (see Lemma 2.3 and Remark 4.5 of \cite{BZ}).

\begin{proposition} Let $u \in W_0^{1,p}(\Omega)$ be a nonnegative function and suppose that $a=a({\rm z})$, $a$ satisfy $(H3)$, $a ({\rm z}) \cdot {\rm z} \in C^2(\mathbb{R}^n \backslash \{0\})$,  $(a({\rm z}) \cdot {\rm z})^{1/p}$ is convex.
If the symmetrization $u^{\sharp}$ is equal to some radial
solution of $({\rm P}_B)$ on $\Omega^{\sharp}_{t_1t_2} = \{ x \in
\Omega^{\sharp} \; : \; t_1 < u^{\sharp}(x) < t_2 \}$ and
$$ \int_{ t_1 < u < t_2 } \nabla u \cdot a(\nabla u ) \; dx \; = \; \int_{ t_1 <
u^{\sharp} < t_2 } \nabla u^{\sharp} \cdot a(  \nabla u^{\sharp} ) \; dx,
$$ for some $0 \le t_1 < t_2 \le \max u < + \infty $, then there is a
translate of $u^{\sharp}$ which is almost everywhere equal to $u$
in $\{ t_1 < u < t_2 \}$. $(({\rm P}_B)$ is the problem $({\rm
\tilde{P}}_B)$ with $\tilde{a}$ replaced by $a.)$
\label{continuidade}
\end{proposition}
\begin{proof} Let $U_1$ be the radial solution of $({\rm P}_B)$ such that
$u^{\sharp} =U_1$ on $\Omega^{\sharp}_{t_1t_2}$. From Lemma \ref{divergencelemma},
$$ \int_{\partial B_t} a( \nabla U_1) \cdot n \; dS = \int_{B_t} f(U_1)
dx > 0 \quad \text{ for any } t \in [0,\max U_1),$$ where $B_t =
\{ x: U_1(x) > t \} $. Hence $a( \nabla U_1) \ne 0$ and, therefore,
$\nabla U_1(x)=0$ for any $x \ne 0$. Then $\nabla u^{\sharp}(x)
\ne 0$ on the closure of $\Omega^{\sharp}_{t_1t_2}$. Since $|\{
\nabla U_1 = 0\} | =0$, according to a result of Brothers
and Ziemer (see Lemma 2.3 and Remark 4.5 of \cite{BZ}),
the equality between the
Dirichlet integrals holds only if $u$ is equal to some translation
of $u^{\sharp}$ almost everywhere on $\{ t_1 < u < t_2 \}$.
\end{proof}

Next we present some comparison results about solutions.

\begin{lemma} Consider the radial functions $u_1(x)=w_1(|x|) \in C^1(B_{R_1})$
and $u_2(x)=w_2(|x|) \in H^1(B_{R_2})$, where $B_{R_i}$ is the
ball centered at $0$ with radius $R_i$, $w_1:[0,R_1]\to
\mathbb{R}$ is decreasing, $w_1'(r) < 0$ for $r >0$, $w_2:[0,R_2]\to \mathbb{R}$ is nonincreasing, and
$R_1> R_2$. Suppose that $m=w_1(R_1)=w_2(R_2)$ and
\begin{equation} \label{integraldenivel} \int_{\{ u_1 = t \}}
\frac{a(u_1, \nabla u_1) \cdot \nabla u_1}{|\nabla u_1|} \; dH^{n-1}
\ge \int_{\{ u_2 = t \}} \frac{a(u_2, \nabla u_2) \cdot \nabla
u_2}{|\nabla u_2|} \; dH^{n-1} \end{equation}  for almost all $t
\in [m,+\infty)$, where $H^{n-1}$ is the $(n-1)$-dimensional
Hausdorff measure and $a=a(t,{\rm z})$ is a function that satisfies (H3). Then $u_1
> u_2$ in $B_{R_2} \backslash \{0\}$.
\label{comparacaoradial}
\end{lemma}
\begin{proof} We prove by contradiction. So there exists some
$r_0 \in (0,R_2)$ such that $w_1(r_0) \le w_2(r_0)$.
The hypotheses imply that $$w_1(R_2) > w_1(R_1)=w_2(R_2).$$ Hence,
from the continuity of $w_1$ and $w_2$, we can assume that
\begin{equation} w_1(r_0)= w_2(r_0) \quad \text{ and } \quad w_1(r) > w_2(r) \quad \text{ for } r \in
(r_0,R_2]. \end{equation} Now defining $b(t,|{\rm z}|)=|a(t,{\rm z})|$, we have
$$ b(u_i, |\nabla u_i|) = |a(u_i, \nabla u_i)| = \frac{ a(u_i, \nabla u_i) \cdot \nabla u_i }{|\nabla u_i|}   \quad \text{for} \;  i=1,2. $$
Observe that $b=b(t,s) \in C^0( \mathbb{R} \times [0,+\infty)) \cap C^1(\mathbb{R} \times (0,+\infty))$ is positive for $s\ne 0$ and increasing in $s$. Hence, using \eqref{integraldenivel} and
$w'_i(|x|)=-|\nabla u_i(x)|$, we get
$$ b(t, - w_1'(r_1(t))) \, r_1^{n-1}(t) \ge b(t, - w_2'(r_2(t))) \, r_2^{n-1}(t)$$ a.e. on $I = [m,t_0]$,
where $t_0=w_1(r_0)=w_2(r_0)$ and $r_i$ is some kind of inverse of $w_i$ given by $r_i(t)=\inf\{r \; | \; w_i(r) \le t\} = (\mu_{u_i}(t)/\omega_n)^{1/n}$. Notice
that $r_1$ is decreasing and $r_2$ is nonincreasing and, therefore, they are differentiable a.e. on $I$ with
$r_i'(t)=(w_i'(r_i(t))^{-1}$. Then
$$ b\left(t, - \frac{1}{r'_1(t)}\right) \, r_1^{n-1}(t) \ge b\left(t, - \frac{1}{r'_2(t)}\right) \, r_2^{n-1}(t) \quad \text{a.e. \, on  \,} I. $$
Defining $d:\mathbb{R} \times (-\infty,0) \to \mathbb{R}$ by $ d(t, y)=
[b(t, -1/y)]^{1/(n-1)}$, we obtain
$$ d(t, r'_1(t)) \; r_1(t) \ge d(t, r'_2(t)) \; r_2(t) \quad \text{a.e. \, on \,} I $$
and, therefore, $$  d(t, r'_1) \; (r_1-r_2) \ge (d(t, r'_2)- d(t, r'_1)) \;
r_2 \quad \text{a.e. \, on \,} I. $$
 Since $r_2 \ge r_0 > 0 $,
$d(t, r'_1(t))$ is continuous and positive in $I$, and $r_1 - r_2 \ge 0$, there exist $c_1 > 0$ such that
\begin{equation} \label{dinequa1} c_1 \;
(r_1-r_2) \ge (d(t, r'_2)- d(t, r'_1)) \; r_0 \quad \text{a.e. \, on \,}
I. \end{equation}
We prove now that, for some suitable constant $C >0$,
\begin{equation} \label{dinequa2} C (r_1-r_2) \ge r'_2 - r'_1  \quad \text{a.e. \, on \,} I.\end{equation}
For that note first that if $t \in I$ satisfies $r'_2(t) \le
r'_1(t)$, the inequality is trivial for any $C > 0$ since $r_1 \ge
r_2$ on $I$. In the case $r'_2(t) > r'_1(t)$,
$$d(t,r'_2)-d(t,r'_1)  = \int_{r'_1}^{r'_2} \frac{ \partial d}{\partial y}(t,y) \, dy  \ge \int_{r'_1}^{\frac{r'_2+r'_1}{2}} \frac{[b\left( t, -\frac{1}{y} \right)]}{n-1}^{\frac{2-n}{n-1}}\cdot \frac{b_{s}\left(t, -\frac{1}{y} \right)}{y^2} \, dy $$
since the integrand is positive and $(r'_2+r'_1)/2 \le r'_2$. From the $C^1$ regularity of $w_1$ and $w_1' < 0$, it follows that the interval $[r_1'(t), r_1'(t)/2]$ is contained in some interval
$[y_1,y_2]$, where $y_2 <0$, for any $t \in I$. Then
$$ [r'_1, (r'_1 + r'_2)/ 2] \subset [r'_1, r'_1/ 2] \subset [y_1,y_2] \subset (-\infty, 0) \quad {\rm for \; any } \; t \in I,$$
and, using that $|a|$ and $\partial_{s}|a(t,s{\rm z})|$ are positive and continuous for $s, {\rm z} \ne 0$,  we get
$$b\left(t, -\frac{1}{y}\right) \ge \min_{[y_1,y_2]} b\left(t, -\frac{1}{y}\right) \ge E_1 := \min_{ \frac{1}{|y_1|} \le |{\rm z}| \le  \frac{1}{|y_2|}} |a(t,{\rm z})| > 0 $$
and
$$ b_s \left(t, -\frac{1}{y}\right) \ge \min_{[y_1,y_2]} b_s \left(t, -\frac{1}{y}\right) \ge E_2 := \min_{ \frac{1}{|y_1|} \le |{\rm z}| \le  \frac{1}{|y_2|}} \partial_{s}|a(t,s{\rm z})|\Big|_{s=1} > 0 $$
for $y \in  [r'_1, (r'_1 + r'_2)/ 2]$.
Hence,
$$ d(t,r'_2)-d(t,r'_1)  \ge  \int_{r'_1}^{(r'_1+r'_2)/2} \frac{E_1^{\frac{2-n}{n-1}}}{n-1} \cdot \frac{E_2 }{y^2} \, dy \ge  \frac{E_1^{\frac{2-n}{n-1}} E_2}{(n-1)\, y_1^2}  \cdot \frac{(r'_2 - r'_1)}{2}. $$
From this and \eqref{dinequa1}, we get \eqref{dinequa2} with $C= 2c_1(n-1)y_1^2/(r_0E_1^{\frac{2-n}{n-1}}E_2)$.
Multiplying \eqref{dinequa2} by $e^{Ct}$,
it follows that
$$ \frac{d  }{dt}(r_1 e^{Ct})  \ge \frac{d  }{dt}(r_2 e^{Ct}) \quad \text{a.e. \, on \,} I.$$
Observe that $\displaystyle \int_m^{t_0} (r_2 e^{Ct})' dt \ge
r_2e^{Ct}\big|_{m}^{t_0}$, since $r_2$ is decreasing and $e^{Ct}$
is a $C^1$ function. To prove that, we can split $r_2 e^{Ct}$ into
a singular function and an absolutely continuous function, apply
the Fundamental Theorem of Calculus, obtaining an identity for the
second part and, using a sequence of increasing $C^1$ functions
that converges uniformly to $r_2$, an inequality for the first
part.
\\ Therefore
$$ r_1e^{Ct}\big|_{m}^{t_0}= \int_m^{t_0} \frac{d  }{dt}(r_1 e^{Ct})
dt   \ge \int_m^{t_0} \frac{d  }{dt}(r_2 e^{Ct})dt \ge
r_2e^{Ct}\big|_{m}^{t_0}.$$ Hence, using $r_1(t_0) = r_2(t_0) =
r_0$, we get $r_1(m) \le r_2(m)$. But this contradicts $r_1(m)=R_1
> R_2 = r_2(m)$.
\end{proof}

\section{Comparison results to the p-laplacian}
\label{SC}

We treat in this section the special case where the differential
part of \eqref{eq1} and $({\rm \tilde{P}}_B)$ is the $p$-laplacian
operator and, in addition to the hypotheses (H1) and (H5), we
suppose that $f(t)/t^{p-1}$ is decreasing. Then, we can obtain a
solution to the problem $({\rm \tilde{P}}_B)$ minimizing the
functional
\begin{equation} J_{B} (v) \; = \; \int_{B} \frac{1}{p} | \nabla v
|^p  - F(v) \; dx, \label{funcional}
\end{equation} where $F(t) = \int_0^t f(s) \; ds $. Let $\tilde{U}_B$ be
a minimum of $J_B$. Since $f(t)/t^{p-1}$ is decreasing,
$\tilde{U}_B$ is the unique solution to $({\rm \tilde{P}}_B)$ (see
\cite{BO} and \cite{BK}). Then $\tilde{U}_B=U_B$, where $U_B$ is
defined in (\ref{definicaoU}). This uniqueness result is applied
only in Theorem \ref{teoremaP}.

\begin{remark}
\label{obs1} For any ball $B_r \subset B$ and $\alpha \in
\mathbb{R}$, there is a radial minimizer $w$ of the functional
$$J_{B_r}(v) \; = \; \int_{B_r} \frac{1}{p} | \nabla v
|^p  - F(v) \; dx, $$ such that $w \equiv \alpha$ on $\partial
B_r$. Moreover, if $u$ and $w$ are minimizers of $J_{B_r}$ and $u
> w$ on $\partial B_{r}$, then $u > w$ in $B_{r}$ and $J_{B_r}(u) < J_{B_r}(w)$.
\end{remark}

\noindent The first part of this remark follows from classical
arguments of compactness. To prove that $u > w$, observe that for
any open subset $A \subset B_{r}$, $u$ and $w$ minimizes the
corresponding functional $J_A$ in the set of functions with
prescribed boundary data $v=u$ and $v=w$ on $\partial A$,
respectively. Hence, if $u < w$ for some open set, then the
$v_0=\max\{u,w\}$ is also a minimizer of $J_{B_r}$ in
$W^{1,p}(B_{r})$ and $u$ touches $v_0$ by below in $B_{r}$,
contradicting $-\Delta_p v_0= f(v_0) \ge f(u) = -\Delta_p u $ and
the maximum principle. The inequality $J_{B_r}(u) < J_{B_r}(w)$ is a
consequence of $J_{B_r}(v) < J_{B_r}(v)$ for any $v \in
W^{1,p}(B_{r})$ since $F$ is increasing. The next result does not
requires that $f$ is nondecreasing.

\begin{theorem} Let $\Omega \subset \mathbb{R}^n$ be a bounded domain,
$B$ be a ball such that $|B|=|\Omega|$, and $u$ be a weak
solution of {\rm (\ref{eq1})}, where $div(a(\nabla u))=\Delta_p u $
and $f$ is a nonnegative locally Lipschitz function, possibly
non-monotone, such that $f(t)/t^{p-1}$ is decreasing on
$(0,+\infty)$. Then,
$$ \max u \; \le \; \max U_B, $$
where $U_B$ is the minimizer of the functional given by {\rm
(\ref{funcional})}. \label{teoremaA}
\end{theorem}

\begin{proof}
Let $u^{\sharp}$ be the Schwarz symmetrization of $u$.
Defining $\Omega_t^{\sharp}=\{ u^{\sharp}
> t \} $, we have that $| \Omega_t^{\sharp} | = |\Omega_t|$.
Therefore, Remark \ref{obs2} implies that
$$ \int_{\Omega_t} F(u) \; dx \; = \; \int_{\Omega_t^{\sharp}} F(u^{\sharp}) \;
dx \quad {\rm for} \; t \ge 0.$$ We also know that
\begin{equation}
 \int_{\Omega_t} | \nabla u |^p \; dx \; \ge \;
\int_{\Omega_t^{\sharp}} | \nabla u^{\sharp} |^p \; dx . \label{S}
\end{equation}
Then, \begin{equation} \int_{\Omega_t} \frac{| \nabla u
|^p}{p} - F(u) \; dx \; \ge \; \int_{\Omega_t^{\sharp}} \frac{|
\nabla u^{\sharp} |^p}{p} - F(u^{\sharp}) \; dx . \label{desig1}
\end{equation}
Now suppose that for some $t \ge 0$, we have $|\Omega_t| = |B_t|$,
where $B_t = \{ U_B > t \}$. In this case $B_t =
\Omega^{\sharp}_t$ and
\begin{equation} \int_{\Omega_t^{\sharp}} \frac{| \nabla u^{\sharp} |^p}{p} - F(u^{\sharp}) \; dx
\; \ge \; \int_{B_t} \frac{| \nabla U_B |^p}{p} - F(U_B) \; dx,
\label{desig2}
\end{equation} otherwise the function $\tilde{u}: B
\to \mathbb{R}$ given by $\tilde{u} = u^{\sharp} \chi_{_{B_t}} +
U_B \chi_{_{B_t^c}} $ is the minimum of $J_B$. Then, from
(\ref{desig1}) and (\ref{desig2}), it follows that
$$ \int_{\Omega_t} \frac{| \nabla u
|^p}{p} - F(u) \; dx \; \ge \; \int_{B_t} \frac{| \nabla U_B
|^p}{p} - F(U_B) \; dx. $$ Hence, using Lemma \ref{divergencia}
and the fact that $u$ and $U_B$ are solutions, we get
\begin{equation} \int_{\Omega_t} \!\!\
\frac{ u f(u) - t f(u)}{p} - F(u) \; dx  \ge  \int_{B_t} \!\!\
\frac{ U_B f( U_B) - t f(U_B)}{p} - F(U_B) \; dx.
\label{DesigualdadeChave}
\end{equation} Define $h_t : [t, + \infty) \to \mathbb{R}$ by
\begin{equation}
 h_t(s) \; = \; \frac{(s-t)f(s)}{p} - F(s).
\end{equation}
Note that $h_t(s)$ is decreasing for $s \ge t$, since
$$ h_t'(s) \; = \; \frac{(s-t)f'(s)}{p} - \frac{(p-1)f(s)}{p}  \; = \;
\frac{(s-t)^{p}}{p}
 \left(\frac{ f(s)}{(s-t)^{p-1}}\right)' \; < \; 0.$$
Furthermore, as $h_t(t)\le 0$, $h_t (s) < 0$ for $s
>t$. Therefore, from (\ref{DesigualdadeChave}), we have
\begin{equation}
\int_{\Omega_t} h_t(u) \; dx \; \ge \; \int_{B_t}h_t(U_B) \; dx ,
\label{desigint}
\end{equation} where
$h_t$ is  decreasing and negative. Suppose that $ \max u
> \max U_B $. Since $|\Omega|=|B|$, the function  $\mu_B(t)=| \{ U_B > t \} |$  is continuous and $\mu_u(t)$
is right continuous, there is $t_0 \ge 0$ such that $\mu_u(t_0) =
\mu_B(t_0)$ and $\mu_u(t) > \mu_B(t)$ for $t > t_0$. Then,
$$ | \{ -h_{t_0} \circ u > s \}| \; > \; | \{ -h_{t_0} \circ U_B > s \} |  \quad {\rm for } \quad s > -h_{t_0}(t_0),$$
since $-h_{t_0}$ is a increasing function. Thus, by Fubini's
Theorem,
$$ -\int_{\Omega_{t_0}} h_{t_0}(u) \; dx \; > \;
-\int_{B_{t_0}}h_{t_0}(U_B) \; dx, $$ contradicting
(\ref{desigint}).   \end{proof}

\begin{remark}
\label{remarkfirst0theorema1}
This result can be extended to the problem \eqref{extensioncaparison} observing first that $q(t):=(f(t)-k(t))/t$ is decreasing.
If $q(t) > 0$ for any $t > 0$, it is immediate from the theorem that $\max u \le \max U$, where $u$ solves \eqref{extensioncaparison}
and $U \in W^{1,p}_0(B)$ is the solution of the symmetrized problem $-\Delta_p V + k(V) = f(V)$ in $B$. If $q(t_0)=0$ for some $t_0 \ge 0$, the maximum principle implies that
$u, U \le t_0$. Hence taking $u_m$ and $U_m$, the sequence of solutions of $-\Delta_p v = \max \{ f(v) -k(v), 0 \} + 1/m$ in $\Omega$ and $B$ respectively,
we have $u_m \le U_m$, $u_m \to u$ and $U_m \to U$ monotonically, proving the inequality. A related result with this one is stated in \cite{D}. For instance, if $f$ is a positive constant,
Theorem 2 of that work give more relations between $u$ and $U$.
\end{remark}

\begin{corollary}
\label{estrita} Assuming the same hypotheses as in Proposition
\ref{teoremaA}, if $\Omega$ is not a ball, then
$$ \max u \; < \; \max U_B \, . $$
\end{corollary}
\begin{proof} If $\Omega$ is not a ball, Proposition
\ref{continuidade} implies that inequalities (\ref{S}) and
(\ref{desigint}) are strict for $t=0$. Therefore, there is $t
> 0 $ such that
$$ |\Omega_t| \; <  \; |B_t|. $$
Note that the function $ v \; = \; u - t $ satisfies
\begin{equation}
\left\{
\begin{array}{rclc}
- \Delta_p v & = & \tilde{f}(v) & {\rm in } \; \; \Omega_t \\
  v & = & 0 & {\rm on } \; \; \partial \Omega_t ,\\ \end{array}
  \right.
\label{solucaot}
\end{equation} where $\tilde{f}$ is given by $\tilde{f}(s) = f(s+t)$.
If $B'$ and $B$ are  concentric balls and $| B'| = |
\Omega_t  |$, then $|B'| < | B_t |$ and $B' \subset
B_t $. Since $U_B = t$ on $\partial B_t$, we get from the maximum
principle that $U_B
> t$ on $\partial B'$. Hence, using Remark \ref{obs1},
there is a function $w: B' \to \mathbb{R}$ that minimizes
$J_{B'}$ under the condition $w \equiv t$ on $\partial
B'$ and, therefore, the function $V_{B'}=w -t$ is
the solution of (\ref{solucaot}) with $\Omega_t$ replaced by
$B'$. Furthermore, $w < U_B$. Since $\tilde{f}$ satisfies
all hypotheses required in Theorem \ref{teoremaA},
$$\max v \; \le \; \max V_{B'} \, .$$
Hence,
$$\max_{\Omega} u \; = \; \max_{\Omega_t} (v + t) \; \le \; \max_{\Omega_t} (V_{B'} + t) \; = \; \max_{\Omega_t}w < \max_{\Omega}U_B $$
proving the result.   \end{proof}

\begin{remark}
\label{remarktheorema1}
Suppose that $u \in W^{1,p}_0(\Omega)$ is a solution of
\begin{equation}-{\rm div} (M Du |Du|^{p-2}) = f(u),
\label{equacaoM}
\end{equation}
where $M(x)=(a_{ij}(x))$ is a matrix with measurable bounded entries such that,
$\sum_{ij}a_{ij}(x) \xi_i\xi_j \ge |\xi|^2$. Observing that
$$  \tilde{J} (v) := \int_{\Omega}\langle MDv,Dv \rangle \frac{1}{p} | \nabla v
|^{p-2}  - F(v) \; dx \ge \int_{\Omega} \frac{1}{p} | \nabla v
|^{p}  - F(v) \; dx, $$
and repeating the arguments of Theorem \ref{teoremaA}, we get $\max u \le \max U_B$.
Notice that $M$ can be nonsymmetric.
\end{remark}

Next result is some sort of maximum principle for the distribution function.
The proof will be given for a more general case in Section \ref{MR}, Proposition \ref{tipodeprincipiodomaximogeral}.

\begin{proposition}
Suppose that $u \in W^{1,p}_0(\Omega)$ and $U \in W^{1,p}_0 (B)$ satisfy $-\Delta_p u = f(u)$ and $-\Delta_p U = f(U)$, where $f$ is a nondecreasing locally Lipschitz
function, positive on $(0,+\infty)$. If $u^{\sharp} \le U$ and $u^{\sharp} \not\equiv U$, then $u^{\sharp} < U$ on $B$.
\label{tipodeprincipiodomaximo}
\end{proposition}

Next theorem, in the case $p=2$ and $f(0) > 0$, is a consequence of a result, which establishes that the symmetrization of the minimal solution associated to $\Omega$ is smaller or equal than the one associated
to the corresponding ball (see \cite{B1}, \cite{L1}), and the uniqueness of solution when $f(t)/t$ is decreasing (see \cite{BO}). For general $p$, we can apply a similar argument to compare the minimal
solutions (see \cite{I}) and the uniqueness result obtained for the case that $f(t)/t^{p-1}$ is decreasing (see \cite{BK}).

Also it can be proved in a independent way using the main result of Section \ref{MR} and the uniqueness of solution to this problem.

\begin{theorem}
\label{teoremaP} Let $\Omega \subset \mathbb{R}^n$ be a bounded
domain, $B$ be a ball such that $|B|=|\Omega|$, and $u$ be a
weak solution of {\rm (\ref{eq1})}, where $div(a(\nabla
u))=\Delta_p u$ and $f$ is a nonnegative increasing locally Lipschitz
function, such that $f(t)/t^{p-1}$ is decreasing on $(0,+\infty)$.
Then,
$$| \{ u > t \} | \; < \; |\{U_B > t \}| \quad \forall t \in
(0,\max U_B],$$
unless $\Omega$ is a ball.
\end{theorem}

\section{Study of the radial solutions}
\label{RC}

We study now a Dirichlet problem, where the domain is a ball, and we need some additional hypothesis:
\\ \\ (H6) there is some $\mu \in [0,2)$ such that $\frac{d}{ds} |a(t,s{\rm w})| \ge |a(t,s{\rm w})|^{\mu}$ for $s > 0$ small and ${\rm w}$ unit vector of $\mathbb{R}^n$.
\\ \\ The following theorem is the main result of this section.

\begin{theorem}
Let $B'=B_{R_0}$ be a open ball in $\mathbb{R}^n$ satisfying $|B'| \le |\Omega|$ and suppose that $\tilde{a}$ and $f$ satisfy
conditions (H1)-(H6).
If $f(0) > 0$ and $m \ge 0$, then there exists a solution $U_{B'}$ to the problem
$({\rm \tilde{P}}_{B'})$ with $U_{B'} = m$ on $\partial B'$ such that, for any radial solution $U$ of $({\rm \tilde{P}}_{B''})$ with $0 \le U \le m$ on $\partial B''$,
$$ U_{B'} > U \quad in \quad B'', $$
where $B'' \subsetneq B'$ are concentric open balls. The same holds
in the case $B'' = B'$ if $U$ and $U_{B'}$ are different.
\label{TeorSolucRadial}
\end{theorem}

\begin{remark} Suppose that the hypotheses of this theorem holds and $U$ is a radial weak solution of (${\rm
\tilde{P}}_{B''}$). We will see that $U$ is a classical solution in $B'' \backslash \{ 0 \}$. First using the ACL characterization of Sobolev functions (see e.g. \cite{Z}) and a local diffeomorphism between the Cartesian and the polar system
of coordinates, it follows that $U$ is absolutely continuous on closed radial segments that does not contain the origin. Hence the set $\{ U < t \}$ is open in $B'$ for any $t \in \mathbb{R}$.
Indeed, these sets are rings of the form $\{ x \in B' : r_t < |x| < R_0 \}$, otherwise there is a ring ${\mathcal R} = \{ r_1 < |x| < r_2 \}$ contained in $\{U < t \}$, such that $U = t$ on $\partial {\mathcal R}$, for which
the test function $\varphi(x)= (t-U(x)) \chi_{{\mathcal R}}(x) \in W^{1,p}_0$ satisfies
$$ 0 \ge -\int_{{\mathcal R}}\nabla U \cdot  \tilde{a}(U,\nabla U)   dx = \int_{{\mathcal R}}\nabla \varphi \cdot  \tilde{a}(U,\nabla U)   dx = \int_{{\mathcal R}} f(U) \varphi dx > 0, $$
that is a contradiction. Hence, $U$ is a nonincreasing radially symmetric function. Observe also that if $U$ is constant in some ring, then taking a nonnegative function with a compact support in this ring, we get
 a contradiction as before. Then $U$ is strictly decreasing in the radial direction. This conclusion can be obtained more easily for operators where the maximum principle holds.

Notice now that for a given ring ${\mathcal R}=\{r_1 < |x| < r_2 \}$, taking the radial test function $\varphi_{R,h}(|x|)=\chi_{[0,R-h]}(|x|) + \left( \frac{R+h}{2h} - \frac{|x|}{2h} \right) \chi_{(R-h,R+h]}(|x|)$, for $h>0$ and $R \in (r_1,r_2)$, we get
$$  n \omega_n \int_{R-h}^{R+h} \frac{b(U, -\partial_r U)}{2h} r^{n-1} dr =  \int_{B'} \nabla \varphi_{R,h} \cdot \tilde{a}(U, \nabla U) dx = \int_{B'} f(U) \varphi_{R,h}  dx,   $$
where $b(t,|{\rm z}|)=|\tilde{a}(t,{\rm z})|$ and  $\omega_n$ is the volume of the unit ball. Making $h \to 0$, from the Lebesgue Differentiation Theorem, it follows that
\begin{equation}
n \omega_n b(U(R), -\partial_r U(R) ) R^{n-1} = \int_{B_R} f(U) dx \ge \int_{B_{r_1}} f(0) dx  > 0
\label{bradialrelation}
\end{equation}
for almost every $ R \in (r_1,r_2)$ and then, using $(H3)$, we get that $|\nabla U| \ge c$ a.e. in ${\mathcal R}$, where $c$ is some positive constant that depends on ${\mathcal R}$.  Thus $U$ is a solution of a
uniformly elliptic equation in this ring and, therefore, a $C^{2, \alpha}$ function in ${\mathcal R}$ for any $\alpha \in (0,1)$.  Moreover, from \eqref{sup3}, $U$ is bounded and, from its monotonicity in the radial direction,
it can be defined continuously on $0$.  In fact, using \eqref{bradialrelation}, we can prove that $U$ is differentiable at the origin and its derivative is zero.

\noindent Due to this regularity of $U$ and $(H3)$, we have $\tilde{a}(U,\nabla U) =
\tilde{e}(U,|\nabla U |) \nabla U$ for some function $\tilde{e}:\mathbb{R} \times [0,+\infty) \to
\mathbb{R}$ and $U(x) = w(|x|)$ for some function $w:[0,R_1]
\to \mathbb{R}$ that satisfies, in the classical sense,
\begin{equation}
\label{radialequation} \left\{ \begin{array}{rcl} \displaystyle (\tilde{e}
\, w')' + \frac{n-1}{r}\; \tilde{e} \,w' & = & -f(w )  \quad {\rm for}
\quad r\in [0,R_1]
\\ w'(0) & = & 0 \\ w(R_1) & = & 0, \end{array} \right.
\end{equation}
where $R_1$ is the radius of $B''$ and $'$ denotes $ d /d
r$. To prove the existence of solution to this problem, we consider the following one:
\begin{equation}
\label{radialequation2} \left\{ \begin{array}{rcl} \displaystyle
(\tilde{e}(w(r), |w'(r)|)\, w'(r))' + \frac{n-1}{r}\tilde{e}\,w'(r) & = & -f(w(r))\\ w'(0)
& = & 0 \\ w(0) & = & h , \end{array} \right.
\end{equation}
where $h > 0$ is given. If $\tilde{e}$ depends only on z, according to Proposition A1 of
\cite{FLS}, there exists $\delta > 0$ and a positive local solution $w_{h}:[0,\delta)
\to \mathbb{R}$ to \eqref{radialequation2}. In the general case, consider first the problem \eqref{radialequation2} with $\tilde{e}$ replaced by $e_0(|{\rm z}|)=\tilde{e}(h,|{\rm z}|)$, that has a local solution $w_0$ defined on $[0,\delta_0)$ as in the previous case. Then, for $k \in \mathbb{N}$, take $\delta_k \le \delta_0$ such that $w_0(r) \ge h-h/k$ on $[0,\delta_k]$ and define $e_k$ such that $\tilde{a}_k(t,{\rm z}) := e_k(t,|{\rm z}|) {\rm z} $ satisfies (H3),(H4),(H6)
and
$$e_k(t,|{\rm z}|)= \left\{ \begin{array}{ll} e_0(|{\rm z}|) = \tilde{e}(h,|{\rm z}|) & \quad {\rm for} \; t \in [h-h/k,+\infty) \\[5pt]
                                             \tilde{e}(t,|{\rm z}|) & \quad {\rm for} \; t \in (-\infty, h-2h/k] .\end{array} \right. $$
Hence $w_0$ is a solution to \eqref{radialequation2} on $[0,\delta_k]$ with $\tilde{e}$ replaced by $e_k$ and, from \eqref{bradialrelation}, $w_0$ is decreasing and $\frac{d w_0}{dr}(\delta_k) \ne 0$. Since (H3) implies that $s\to |\tilde{a}_k(t,s{\rm w})|$ is increasing for any ${\rm w}$, the classical ODE theory implies that we can extend $w_0$ for a larger interval.
 Indeed, while some extension is positive, it can be continued to
a bigger interval. Since $f(0) > 0$ and $\tilde{a}_k$ satisfies (H4), integrating
$(r^{n-1}e_k(w(r), |w'(r)|)w'(r))'=-r^{n-1}f(w(r))$, we conclude that for any positive
continuation $\bar{w}_{k}:[0,\bar{\delta}) \to \mathbb{R}$ of
$w_{0}$, the right end point satisfies
\begin{equation} \bar{\delta} \le C
:= \frac{nC^*}{f(0)} \left[ 1+ \left( \frac{p}{p-1}\cdot \frac{hf(0)}{nC^*} \right)^{\frac{p-1}{p}} \right].
\label{Rbound0}
\end{equation}

\noindent Hence, there exists a continuation
$w_k:[0,R_{k}] \to \mathbb{R}$ such that
$w_k(R_{k})=0$ and is positive on $[0,R_{k})$. Observe now that, using the same idea as in the estimate \eqref{bradialrelation},
we get that $|w_k'|$ is uniformly bounded by above. Hence some subsequence converge uniformly for some nondecreasing function $w_h:[0,R_h]\to \mathbb{R}$ that is positive in $[0,R_h)$ and vanishes at $R_h$.
Indeed, applying again a similar computation as in \eqref{bradialrelation} and using the positivity of $|\partial_{s}\tilde{a}_k(t,s{\rm z})|$ for $s,t\ne 0$ from $(H3)$,
it follows that $w_k'$ are are equicontinuous in compacts sets
of $[0,R_h)$ for $k$ large. (More precisely, the Lipschitz norm of $w_k'$ are uniformly bounded in compacts sets of $(0,R_h)$ and $w_k'(r)$ are uniformly close to $0$ for $r$ small.)
Hence, some subsequence converge uniformly for $w_h$ in the $C^1$ norm for compact sets of $[0,R_h)$. Hence, due to the regularity of $\tilde{a}$ and the definition of $\tilde{a}_k$, $U_h(x):=w_h(|x|)$
is the weak solution of $-{\rm div }\; \tilde{a}(v, \nabla v) = f(v)$ in $B_{R_h}$. Then, as we observed previously, $U_h$ is a classical solution, and satisfies $U_h(0)=h$ since $w_k(0)=h$.
Moreover, following the same argument of
Proposition A4 of \cite{FLS} for $\tilde{a}$ that depends also on $t$, for each $h>0$, such solution $U_h$ and
radius $R_h$ are unique. Let us represent this correspondence by
$\Psi=(\Psi_1,\Psi_2)$, where $\Psi_1(h) = R_h$ and $\Psi_2(h) =
U_{h}$. \label{AssociacaoAlphaR}
\end{remark}

Observe that $R_{h} \le C$, where $C$ is given by \eqref{Rbound0}.
Using this, the equicontinuity of the first derivative of solutions, Arzel\`a-Ascoli
Theorem and uniqueness for \eqref{radialequation2}, we get the following result.

\begin{lemma} The function $\Psi_1$ is continuous on
$(0,+\infty)$. Furthermore, for any $h_0 >0$, $\varepsilon
>0$ and $K$ compact subset of $B_{R_{h_0}}$, there exists
$\delta
>0$ such that
$ \|\Psi_2(h) - \Psi_2(h_0) \|_{C^1(K)} \le \varepsilon \quad if
\quad | h - h_0| < \delta .$ \label{ContinuidadePsi}
\end{lemma}

We can also improve estimate \eqref{Rbound0} in the following sense.

\begin{lemma}
\label{littleimprovment}
Given $M > 0$, there exists some continuous increasing function $\Theta_M:[0,M]\to \mathbb{R}$
s.t. $\Theta_M(0)=0$ and $R_h \le \Theta_M(h)$ for $h \le M$, where $R_h =\Psi_1(h)$, i.e., $R_h$ is the point s.t. the nonnegative solution $w$ of \eqref{radialequation2} vanishes.
 \end{lemma}
\begin{proof}
 Integrating $(r^{n-1}e(w(r),| w'(r)|)w'(r))'=-r^{n-1}f(w(r))$ from $0$ to $R \le R_h$, we get
$$\left|\tilde{a}\left(w(R),|w'(R)|{\rm z}\right) \right| = - e(w(R),|w'(R)|)w'(R) \ge \frac{f(0)R}{n} $$
for any $|{\rm z}| =1$. Since $s\mapsto |\tilde{a}(t,s{\rm z})|$ is continuous, strictly increasing in $[0,+\infty)$ and vanishes at $s=0$, where $t \in [0,M]$, the function $\rho(s):= \sup_{t\in [0,M]} |\tilde{a}(t,s{\rm z})|$ also
satisfies these hypotheses. Using that $w(R) \le h \le M$,
$$ \rho (-w'(R)) \ge \frac{f(0)R}{n}. $$
Taking the inverse of $\rho$ and integrating from $0$ to $R_h$,
$$ h= w(0)- w(R_h) = \int_0^{R_h }-w'(R) \; dR  \ge \int_0^{R_h} \rho^{-1} \left( \frac{f(0)R}{n} \right) \; dR .  $$
Observe that
$$ R_h \mapsto \int_0^{R_h} \rho^{-1} \left( \frac{f(0)R}{n} \right) \; dR $$
in invertible, since is increasing, positive and vanishes at $0$. Hence, we get the result defining $\Theta_M$ as the inverse of this application.
\end{proof}

\begin{lemma} Assuming the same hypotheses as in Theorem
\ref{TeorSolucRadial}, there exists a solution $U_{B'}$ to the
problem $({\rm P}_{B'})$ with $U_{B'}=0$ on $\partial B'$, such that
$$ \max U_{B'} \ge \max U ,$$
for any radial solution $U$ of $({\rm P}_{B''})$ satisfying $U=0$ on $\partial B''$, where $B''
\subset B'=B_{R_0}$ are concentric balls. As a matter of fact, $U_{B'}=
\Psi_2(h_0)$, where $h_0= \max \{h \; | \; \Psi_1(h) = R_0 \}$. Furthermore, the inequality is strict if $U \ne U_{B'}$.
\end{lemma}
\begin{proof} First we note that Lemma \ref{littleimprovment} implies
that
$$\Psi_1(h_1) = R_{h_1} \le \Theta_1(h_1) < R_0 \quad {\rm for \; small} \; h_1 , $$  since $\Theta_1(h) \to 0$ as $h \to 0$.
We can also prove that $\Psi_1(h_2) > R_0$ for a large $h_2$.
Indeed, from \eqref{sup3}, any solution of $({\rm P}_{B''})$
is bounded by $C|B'|^{1/q}$ if $n < q$ or by $C|B'|^{1/n}$ if $n
\ge q$. Hence,
\begin{equation}\Psi_1(h) > R_0 \quad {\rm for} \quad h >
M = \max \{C|B'|^{1/q}, C |B'|^{1/n}\}, \label{alphabound}
\end{equation}
otherwise a ball of radius $\Psi_1(h) \le R_0$ posses a solution
of height $h > M$ contradicting \eqref{sup3}.

Thus, from the continuity of $\Psi_1$, the set $A=\{h \; | \;
\Psi_1(h) = R_0 \}$ is not empty and is bounded by $M$. Then, we
can define $h_0= \max A$ and $U_{B'}= \Psi_2(h_0)$. Let $U$ be a
radial solution of $({\rm P}_{B''})$ satisfying $U=0$ on $\partial B''$, where
$B''=B_{\tilde{R}}$ with $\tilde{R} \le R_0$. Note that
$\tilde{R}=\Psi_1(U(0))$ and, thus, inequality (\ref{alphabound})
implies that $U(0) \le M$. To prove the lemma we have to show that
$U(0)\le h_0$. Suppose that $U(0) > h_0$. For $h=M +1$, we have
$\Psi_1(h)
> R_0$ from (\ref{alphabound}). Summarizing,
$$ \Psi_1(U(0)) = \tilde{R} \le R_0 < \Psi_1(h) \quad {\rm and } \quad U(0) < h.$$
Therefore, from the continuity of $\Psi_1$, there exists $h_1 \in
[U(0),h)$ such that $\Psi_1(h_1)= R_0$. But this contradicts $h_1
\ge U(0) > h_0$ and the definition of $h_0$. Hence $U(0) \le h_0$.
Furthermore, the equality happens only if $U=U_{B'}$, since the
solution of (\ref{radialequation2}) is unique.
\end{proof}

\begin{proof} {\it of Theorem \ref{TeorSolucRadial}}
\\ \underline{Possibility 1}: $m= 0$
\\ Let $U_{B'}$ be the function defined in the previous lemma and $U$ a solution of
$({\rm P}_{B''})$ with $U=0$ on $\partial B''$, where $B'' \subset B'$ are concentric balls. The set
$$ C = \{ h >0 \; | \; w_{h}:=\Psi_2(h) \ge U_{B'} \; {\rm in} \; B' \; {\rm and} \; w_{h} \ge U \; {\rm in} \; B'' \} $$
is not empty. To prove that, let $h > \max U_{B'}$ such that $h
\not\in C$. For instance, suppose that $w_{h}$ does not satisfy
$w_{h} \ge U_{B'}$ in $B'$. Using that $w_{h}$ and $U_{B'}$ are
continuous radial functions and $w_{h}(0)=h > U_{B'}(0)$, we conclude
that there exists $B'' \subset B'$ such that $w_{h} > U_{B'}$ in $B''$
and $w_{h}=U_{B'}$ in $\partial B''$. Denoting $t_0=U_{B'}^{-1}(B'')$, we have $t_0 \le \max U_{B'} \le M$, where $M$ is given by \eqref{alphabound}.
Hence, the function $\tilde{f}(t)=f(t+t_0)$ satisfies $$\tilde{f}(t) \le f(t+M) \le \alpha(t+M)^{q-1}+ \beta \le \alpha't^{q-1} + \beta',$$
where $\alpha'$ is any real in $(\alpha, C_* \lambda_{B} )$ and $\beta'$ is a constant that depends on $\alpha'$, $\beta$ and $M$.
Note that $v= w_h - t_0$ satisfies $$-{\rm div}(\bar{a}(v,\nabla v))= \tilde{f}(v),$$ where $\bar{a}(t,{\rm z})=\tilde{a}(t+t_0,{\rm z})$, with the boundary data $v=0$ on $B''$. Since $\bar{a}$ and $\tilde{f}$ satisfy
(H1)-(H6), it follows from \eqref{sup3} that
$\sup v \le \tilde{M}$, where $\tilde{M}$ is a constant that depends on $n$, $q$, $\alpha'$, $\beta'$, $C_*$, and $|\Omega|$. Thus
$w_h \le \tilde{M} + M$. This inequality also holds, by the same argument, when condition $w_{h} \ge U$ in $B'$ is not satisfied.
Therefore, $h \in C$ for $h > \tilde{M} + M$, proving that $C$ is not empty.

Let $\alpha_1 = \inf C$. From the continuity of $\Psi_1$ and the
$C^1$ estimate of Lemma \ref{ContinuidadePsi},
$R_1=\Psi_1(\alpha_1) \ge R_0$, $w_{\alpha_1}=\Psi_2(\alpha_1) \ge
U_{B'}$ in $B'$, and $w_{\alpha_1} \ge U$ in $B''$. Hence
$\alpha_1=w_{\alpha_1}(0) \ge U_{B'}(0)$. If
$\alpha_0:=U_{B'}(0)=\alpha_1$, then $w_{\alpha_1}=U_{B'}$ and, therefore, $U_{B'}
\ge U$ proving the theorem. Suppose that $\alpha_1 > \alpha_0$.
Then $R_1 > R_0$, otherwise $R_0=R_1=\Psi_1(\alpha_1)$
contradicting $\alpha_1 > \alpha_0 = \max \{\alpha \; | \;
\Psi_1(\alpha) = R_0 \}$. Let
$$ d_1= \inf_{x \in B'} (w_{\alpha_1}(x) - U_{B'}(x))\ge 0 \quad {\rm and}
\quad d_2= \inf_{x \in B''} (w_{\alpha_1}(x) - U(x)) \ge 0.
$$ If $d_1=0$, consider $x_1 \in \bar{B'}\backslash \{ 0 \}$ such that $w_{\alpha_1}(x_1) =
U_{B'}(x_1)$. Since $R_1 > R_0$, we have $w_{\alpha_1} >0$ in
$\partial B'$ and, from $U_{B'}=0$ in $\partial B'$, it follows
that $x_1 \in B'\backslash \{ 0 \} $. Observe also that $\nabla w_{\alpha_1}(x_1) = \nabla U_{B'}(x_1) $,
since $w_{\alpha_1} \ge U_{B'}$.  Then, using that $w_{\alpha_1}$ and $U_{B'}$ are radial, we infer from the uniqueness of solution for ODE that $w_{\alpha_1}=U_{B'}$, contradicting
$w_{\alpha_1}(0)=\alpha_1 > \alpha_0=U_{B'}(0)$. Hence $d_1 > 0$
and, by the same argument, $d_2 >0$. These contradict Lemma \ref{ContinuidadePsi} and the definition of $\alpha_1$, proving that $U_{B'} \ge U$.
\\ \underline{Possibility 2}: $m > 0$
\\Consider the equation
$$ -{\rm div } \; \bar{a} (V, \nabla V) = \tilde{f}(V), $$
where $\bar{a}(t,{\rm z})=\tilde{a}(t+m,{\rm z})$ and $\tilde{f}(t)=f(t+m)$. Notice that $\bar{a}$ and $\tilde{f}$ satisfy (H1)-(H6) with the constants $n$, $p$, $q$, $q_0$, $\alpha'$, $\beta'$, $C_*$,
$C^*$, $C_s$ and $|\Omega|$, where $\alpha'$ and $\beta'$ can be chosen, as in Possibility 1, s.t. $\alpha' \in (\alpha, C_* \lambda_{B} )$ and $\beta'=\beta'(\alpha',\beta,m)$. Then, from Possibility 1, let $\tilde{U}\in W_0^{1,p}(B')$ be the maximal solution associated to this equation.
If $U$ is a solution of $({\rm \tilde{P}}_{B''})$ with $U \le m$ on $\partial B''$, then  $U-m \le 0$ or $U-m$ is also a solution of this equation in some ball contained in $B''$.
In both situations, since $\tilde{U}$ is maximal, $\tilde{U} \ge U-m$. So we conclude Possibility 2, taking $U_{B'}=\tilde{U} + m $.

To prove the strict inequality in case $U \not\equiv U_{B'}$, we must observe that if $U(x_0) = U_{B'}(x_0)$ at some $x_0 \in B''$, then $\nabla U(x_0)=\nabla U_{B'}(x_0)$ since $U \le U_{B'}$.
This contradicts the classical results of uniqueness of solution for ODE if $x_0\ne 0$ and the uniqueness established by Proposition A4 of \cite{FLS} if $x_0=0$, as we already pointed out.
\end{proof}

\begin{theorem}
Let $B'=B_{R_0}$ be a open ball in $\mathbb{R}^n$ satisfying $|B'| \le |\Omega|$ and suppose that $\tilde{a}$ and $f$ satisfy
conditions (H1)-(H6).
If $f(0) = 0$ and $m \ge 0$, then there exists a nonnegative solution $U_{B'}$ of
$({\rm \tilde{P}}_{B'})$  with $U_{B'}=m$ on $\partial B''$, possibly null, s.t. for any radial solution $U$ of $({\rm \tilde{P}}_{B''})$ with $U\le m$ on $\partial B''$,
$$ U_{B'} \ge U \quad in \quad B'', $$
where $B'' \subset B'$ are concentric open balls. If $U_{B'}$ is not trivial, then $U_{B'}$ is positive and the inequality is strict unless $U$ and $U_{B'}$ are equals.
\label{TeorSolucRadialfnull}
\end{theorem}

\begin{proof}
Let $(t_k)$ be a sequence of positive reals s.t. $t_k \downarrow 0$,
$f_k(t):=f(t+t_k+m)$ and $a_k(t,{\rm z}):=\tilde{a}(t+t_k+m, {\rm z})$. Since $a_k$ and $f_k$ satisfy (H1)-(H6) and $f_k(0)=f(t_k+m) > 0$, we can apply Theorem \ref{TeorSolucRadial} to obtain the maximal solution $U_k \in W_0^{1,p}(B')$ of
\begin{equation}
 -{\rm div} \, a_k(v, \nabla v) = f_k(v)
\label{equacaofnull}
\end{equation}
in $B'$. Observe that if $U$ is a radial solution of $({\rm \tilde{P}}_{B''})$ satisfying $0 \le U \le m$, then $U - t_k -m \le 0$ or $U - t_k -m $ is also a solution of \eqref{equacaofnull}
in a ball contained in $B''$ vanishing on the boundary of this ball.
Then, $U_k > U -t_k-m$. Furthermore, since the important constants ($n$, $q$, $\alpha'$, $\beta'$, $C_*$, $|\Omega|$) associated $a_k$ and $f_k$  can be chosen not depending $k$,  $U_k$ is bounded in the $L^{\infty}$ norm by the same argument as in Theorem \ref{TeorSolucRadial}. Therefore, following the estimates of Remark \ref{AssociacaoAlphaR} we get that $\nabla U_k$ is a family of equicontinuous functions. Hence, for some subsequence that we denote by $U_k$, it follows that $U_k$ converges to some
function $U_0$ in the $C^1$ norm. Therefore, $U_{B'}:=U_0 +m $ is a solution of $({\rm \tilde{P}}_{B'})$, with $U_{B'}=m$ on $\partial B'$, and $U_{B'} \ge U$, proving the first part.

Suppose now that $U_{B'}$ is not trivial. According to Remark \ref{AssociacaoAlphaR}, $U_{B'}=w_0(|x|)$ for some nonnegative nonincreasing function $w_0: [0,R_0] \to \mathbb{R}$.
If $w_0(r^*)=0$ for some $r^* \in [0,R_0)$, then $w'(r^*)=0$ since $w$ is differentiable. But, this contradicts Lemma \ref{divergencelemma} and the fact that $f(U_{B'})$ is positive in some nontrivial set. Then $U_{B'}$ is positive
in $B'$. If $U$ is a radial solution in $B''$ different from $U_{B'}$, then these functions cannot be equal at some point, otherwise $U$ touches $U_{B'}$ by below contradicting the uniqueness of solution for ODE.
\end{proof}

\begin{remark}
 If (H6) is not satisfied in Theorem \ref{TeorSolucRadial} or \ref{TeorSolucRadialfnull}, we still have the existence of $U_B$ such that $U_B \ge U$, as we will see in the next section
as a particular case of the main theorem. However, we cannot guarantee the strict inequality. Maybe it is possible that $U_B(0)=U(0)$ and $U_B \not\equiv U$, since (H6) is important for uniqueness of solution for \eqref{radialequation2}.
\end{remark}

\section{Estimates for sublinear equations}
\label{MR}

\begin{proposition}
Let $\Omega \subset \mathbb{R}^n$ be a bounded
open set, $B$ be the ball centered at the origin with
$|B|=|\Omega|$, and suppose that $a$ and $f$ satisfy hypotheses $(H1)$-$(H5)$ and $\tilde{a}$ satisfies $(H3)$-$(H4)$, possibly with different constants $(\tilde{C}_s, \tilde{C}_*, \tilde{C}^*)$ and different powers $(\tilde{p}, \tilde{q}, \tilde{q}_0)$. Assume also that
$\tilde{a}(t,{\rm z})\cdot {\rm z} \le a(t,{\rm z})\cdot {\rm z} $ for any $ {\rm z} \in \mathbb{R}^n$ and  $\tilde{a}(t, {\rm z}) \cdot {\rm z} = \tilde{C}_s|{\rm z}|^{\tilde{q}_0}$ for $|{\rm z}|< \delta$, where $\delta \in (0,1)$.
Then, there exists a radial solution $U_B \in W_0^{1,p}(B)$ of $({\rm \tilde{P}}_B)$ s.t. for any solution $u$ of $({\rm P}_{\Omega})$,
$$ U_B \ge u^{\sharp} \quad {\rm in } \; \Omega^{\sharp}. $$
\label{primeiracomparacao}
\end{proposition}

\begin{remark}
\label{aStarAuxiliar}
There exists a function $a^*({\rm z}) \in C^0( \mathbb{R}^n; \mathbb{R}^n ) \cap C^1(\mathbb{R}^n \backslash \{0\}; \mathbb{R}^n)$, of the form $a^*({\rm z}) =  b^*(|{\rm z}|){\rm z}/|{\rm z}|$, where $b^* \in C^1(\mathbb{R} \backslash \{0\})$ is  positive on $\mathbb{R} \backslash \{0\}$,
$a^*(0)=0$, $a^*({\rm z})\cdot {\rm z}$ is convex, that satisfies
\\ $\bullet$  $|a^*| \le |\tilde{a}|$,
\\ $\bullet$ $a^*({\rm z}) \cdot {\rm z} = \tilde{C}_s|{\rm z}|^{\tilde{q}_0}$ for $|{\rm z}|< \delta$,
\\ $\bullet$ $a^*({\rm z}) \cdot {\rm z} \ge \eta \, \tilde{C}_*|{\rm z}|^{\tilde{q}}$ for $|{\rm z}| \ge 1$, where $\eta \in (0,1)$,
\\ $\bullet$ $a^*({\rm z}) \cdot {\rm z} = \eta \, \tilde{C}_*|{\rm z}|^{\tilde{q}}$ for ${\rm z}$ large.
\\ For that, define $b^*$ in $[0,\delta]$ by $b^*(s)=\tilde{C}_s s^{\tilde{q}_0-1}$.
Then, extend $s \: b^*(s)$ linearly to $[\delta,1]$ in such a way that it is $C^1$ in $[0,1]$. Defining $a^*({\rm z})= b^*(|{\rm z}|) {\rm z}/|{\rm z}|$, we have that $|a^*| \le |\tilde{a}|$
in $B_1(0)$ from the convexity of $\tilde{a}(t,{\rm z})\cdot {\rm z}$. Let $h=b^*(1)$ and $\eta' < \min \{ 1, h/C^* \}$.  Hence, $s \: b^*(s) |_{s=1} >  \eta' \, \tilde{C}_* s^{\tilde{q}}|_{s=1}$ and we can extend $s \: b^*(s)$ linearly until the graph $(s,s\: b^*(s))$ reaches $(s, \eta' \tilde{C}_*s^{\tilde{q}})$ at some point $s_0$. So define $b^*(s)$ that satisfies
$ s\: b^*(s) < \eta' \tilde{C}_*s^{\tilde{q}}$ for $s > s_0$, $s \: b^*(s)$ is convex and $s \: b^*(s) = \eta' \tilde{C}_*s^{\tilde{q}}/2$ for $s$ large. Taking $\eta=\eta'/2$, the function $a^*({\rm z})$ defined from $b^*$ as before, fulfills the requirements.
\end{remark}

\begin{lemma}
\label{SolucaoNoPico}
Assume the same hypotheses as in the previous proposition
and that $u$ is a solution of $($\ref{eq1}$)$. Then there exists
 $t_0 \le \sup u$, an open ball $B^*$ centered at $0$ with the same measure as
$\{u \ge t_0 \}$, and a radial solution $U_{t_0}$ for
\begin{equation}  \left\{
\begin{array}{rcll} - {\rm div } \; \tilde{a} (V, \nabla V) & = & f(V) &
\quad {\rm in} \quad B^* \\[2pt]
V  & = & t_0 & \quad {\rm on} \quad \partial B^*
\end{array} \right. \label{ut0} \end{equation}
such that $U_{t_0} \ge u^{\sharp}$ in $B^*$.
\end{lemma}
\begin{proof} Let $M=\text{ess sup } u >0$, that is finite by Lemma \ref{moserLplemma}.
\\ \underline{Possibility 1}:  $|\{ u = M\}| > 0$
\\ Let $r_0$ be such that the ball $B^*=B_{r_0}(0)$ has the same measure as $\{u = M\}$.
Applying Theorem \ref{TeorSolucRadial} or Theorem \ref{TeorSolucRadialfnull} for $B'=B_{r_0}$ and $m=M$, there exists some maximal solution $U_{B'}$ for $\eqref{ut0}$ with $t_0=M$.
Then, the result follows taking $t_0=M$ and $U_{t_0}(x)=U_{B'}$.
\\ \underline{Possibility 2}:  $|\{ u = M\}| = 0$
\\ Since $f$ is locally Lipschitz and
positive in some neighborhood of $M$, there exists some
$\varepsilon_0 >0$ such that, for any $\varepsilon \le
\varepsilon_0$, the function
$$ G_{\varepsilon}(t) := \frac{f(t)}{(t-(M-\varepsilon))^{\tilde{q}_0-1}} $$ is decreasing on
$(M-\varepsilon,M+\varepsilon_0)$.
\\ {\it Part 1:} For $\varepsilon' \le \varepsilon_0$ small and $t_1 \in (M-\varepsilon',M)$, there is a solution $U_{t_1}$ to the problem \eqref{ut0} with $t_0$ replaced by $t_1$ such that  $|\{ U_{t_1} > t_1 \}| = \mu_u(t_1)$, $\sup U_{t_1} < M +\varepsilon_0$ and $|\nabla U_{t_1}| \le \delta$, where $\delta$ is given in Proposition \ref{primeiracomparacao}.

\

\noindent To prove this, observe that the definition of $M$ implies that $\mu_u(t) > 0$ for $t \in (M-\varepsilon_0,M)$. For $t_1 \in (M-\varepsilon_0,M)$, let
$r_1$ be such that the ball $B_{r_1}(0)$ satisfies $|B_{r_1}(0)|=\mu_u(t_1)$.
Using the same argument as in the Possibility 1, there exists a radial solution $U_{t_1}$ for \eqref{ut0} with $t_0$ and $B_{r_0}$ replaced by $t_1$ and $B_{r_1}$.
We have that $U_{t_1}-t_1$ is a solution of $$-{\rm div \; } \bar{a}(U,\nabla U) = \tilde{f}(U),$$
where $\bar{a}(t,{\rm z})=\tilde{a}(t+t_1,{\rm z})$ and $\tilde{f}(t)=f(t+t_1)$, that vanishes on $\partial B_{r_1}(0)$.
Note that $\bar{a}$ and $\tilde{f}$ satisfy (H1)-(H6) (the constants associated to $\tilde{f}$ are $\alpha' \in (\alpha, \tilde{C}_* \lambda_B)$ and $\beta'$ as in the proof of Theorem \ref{TeorSolucRadial}). Hence, \eqref{sup3} implies that $\sup U_{t_1}-t_1 \le C |B_{r_1}(0)|^{\sigma}$, where $C=C(n,\tilde{q},\alpha',\beta', \frac{\eta \tilde{C}_*}{\tilde{q}_0},|\Omega|) >0$, $\eta$ is associated to $a^*$ from Remark \ref{aStarAuxiliar}, and $\sigma=1/q$ if $q> n$ or $\sigma=1/n$ if $q \le n$. (Since $\eta \in (0,1)$ and $\tilde{q}_0 > 1$, any operator $\bar{a}$ satisfying $\bar{a}(t,{\rm z})\cdot {\rm z} \ge \tilde{C}_* |{\rm z}|^{\tilde{q}}$ also satisfies $\bar{a}(t,{\rm z})\cdot {\rm z} \ge \frac{\eta \tilde{C}_*}{\tilde{q}_0} |{\rm z}|^{\tilde{q}}$. Thus we can consider $C=C(n,\tilde{q},\alpha',\beta', \frac{\eta \tilde{C}_*}{\tilde{q}_0},|\Omega|) \ge C_1:=C_1(n,\tilde{q},\alpha',\beta', \tilde{C}_*,|\Omega|)$ and we can take $C$ instead $C_1$.) Therefore,
$$ \sup U_{t_1} \le C (\mu_{u}(t_1))^{\sigma} + t_1 \le C  (\mu_{u}(t_1))^{\sigma} + M. $$
For $\varepsilon_1 \le \varepsilon_0$ that will be defined later, since
$$\lim_{t \to M^-} \mu_u (t)=|\{ u = M \}| = 0,$$
we get $(\mu_u(t))^{\sigma} < \varepsilon_1/C$ for $t \in (M-\varepsilon',M)$, where $\varepsilon' \le \varepsilon_0$ is small enough.
Thus, $\sup U_{t_1} < M + \varepsilon_0$.
For $t \ge t_1$, define $r(t)$ such that $\partial B_{r(t)}(0) = \{U_{t_1} = t\}$. Then, in the case $|\nabla U_{t_1}(x)| \le 1$, (H4) and Lemma \ref{divergencelemma} imply that
\begin{align*}
n \omega_n r(t)^{n-1} \tilde{C}_s |\nabla U_{t_1}(x)|^{\tilde{q}_0} &\le  \int_{\partial B_{r(t)}} |\tilde{a}(U_{t_1},\nabla U_{t_1})|  dH^{n-1} \\
                                                &= \int_{B_{r(t)}(0)} f(U_{t_1}) dx \le  \omega_n r(t)^n  f (M+ \varepsilon_0),
\end{align*}
for $x \in \{U_{t_1} = t \}$. From this estimate and $|B_{r(t)}| \le |B_{r_1}| = \mu_u(t_1) < (\varepsilon_1/C)^{\frac{1}{\sigma}}$,
$$ |\nabla U_{t_1}(x)| \le \left( \frac{\varepsilon_1}{C \omega_n^{\sigma}} \right)^{\frac{1}{\sigma n\tilde{q}_0 }} \left( \frac{f(M+\varepsilon_0)}{n\tilde{C}_s} \right)^{\frac{1}{\tilde{q}_0}} \quad {\rm for \; } x \in B_{r_1}(0). $$
In the case $|\nabla U_{t_1}(x)| > 1$, a similar estimate holds replacing $\tilde{C}_s$ by $\tilde{C}_*$ and $\tilde{q}_0$ by $\tilde{q}$. Any way,
taking $\varepsilon_1$ small, $ |\nabla U_{t_1}(x)| \le \delta$, where $\delta$ is given in hypothesis of Proposition \ref{primeiracomparacao}.
Therefore, $U_{t_1}$ satisfies the $\tilde{q}_0$ laplacian equation
\begin{equation}
\label{q0tildelaplaceequation}
  - \tilde{C}_s \Delta_{\tilde{q}_0} U_{t_1} = f(U_{t_1}) \quad {\rm in } \; B_{r_1}.
\end{equation}
{\it Part 2:} $U_{t_1}$ is the minimizer of the functional
$$ I_{t_1} (V):= \int_{B_{r_1}}   \frac{ \nabla V  \cdot \tilde{a}(V,\nabla V)}{\tilde{q}_0} - \bar{F}(V) \; dx   $$
in the space $E= \{ V \in W^{1,\tilde{q}}(B_{r_1}) \; | \; V=t_1 \; {\rm on } \; \partial B_{r_1} \}$, where $\bar{F}(t) = \int_0^t\bar{f}(s) ds$,
$$\bar{f}(s) = \left\{ \begin{array}{cc} f(s) & {\rm if}  \; s \le M +\varepsilon_0  \\[2pt]
                                         f(M + \varepsilon_0) & {\rm if} \; s > M +\varepsilon_0.
                       \end{array}
              \right. $$
For that, consider $a^*$ with the properties stated in the Remark \ref{aStarAuxiliar}.
Therefore,
$$ I_{t_1}^*(V)  \le I_{t_1}(V) \quad {\rm for} \quad V \in E, $$
where $I_{t_1}^*$ is defined replacing $\tilde{a}$ by $a^*$ in the definition of $I_{t_1}$. From the growth conditions on $a^*$ and $\bar{f}$, we can use standarts techniques to prove that $I_{t_1}^*$ has a global minimum $U^* \in E$.
Moreover, this minimum is a solution of $$ -{\rm div } \; \hat{a}(\nabla V) = \bar{f}(V) \quad {\rm in } \quad B_{r_1}, $$
where $$\hat{a}({\rm z}):= \frac{a^*({\rm z}) + {\rm z} \cdot D a^*({\rm z})}{\tilde{q}_0}. $$
Observe that $\hat{a}({\rm z})\cdot {\rm z} \ge a^*({\rm z}) \cdot {\rm z}/\tilde{q}_0$ since $s \mapsto |a^*(s {\rm z})|$ is increasing from (H3).
Hence $\hat{a}$ and $\bar{f}$ satisfy (H1), (H5), $ \eta \tilde{C}_{*}/\tilde{q}_0 (|{\rm z}|^q -1) \le \hat{a}({\rm z}) \cdot {\rm z} $ for ${\rm z} \in \mathbb{R}^n$, $t \in \mathbb{R}$ where the important constants in order to apply \eqref{sup3} are $n$, $\tilde{q}$, $\alpha$, $\beta$, $\eta \tilde{C}_*/\tilde{q}_0$ and $|\Omega|$.
Then, as in Part 1, $\sup U^* -t_1 < C|B_{r_1}(0)|$, where $C= C(n,\tilde{q},\alpha',\beta', \frac{\eta \tilde{C}_*}{\tilde{q}_0},|\Omega|)$ is the same constant as before. (Now it is clear why we chose a constant $C$ depending on $\eta\tilde{C}_*/\tilde{q}_0$ instead of $\tilde{C}_*$ at that moment.) Thus
$ \sup U^* < M + \varepsilon_0$ and, following the same computations as before, $|\nabla U^*| < \delta$. Then, from $a^*(t,{\rm z}) = \tilde{a}(t,{\rm z}) $ for $|{\rm z}| < \delta$, it follows that
$$ I_{t_1}^*(U^*) = I_{t_1}(U^*)$$
and, therefore, $U^*$ is also a global minimizer of $I_{t_1}$. From $a^*(t,{\rm z}) = \tilde{C}_s|{\rm z}|^{\tilde{q}_0-2}{\rm z} $ for $|{\rm z}| \le \delta$, we have that $U^*$ is also a solution of \eqref{q0tildelaplaceequation}.
Hence $U_{t_1}-t_1$ and $U^*-t_1$ are solutions of $-\tilde{C}_s \Delta_{\tilde{q}_0} U = \tilde{f}(U)$. Taking $\varepsilon=M-t_1$, we have that $\tilde{f}(t)/t^{\tilde{q}_0}=G_{\varepsilon}(t+t_1)$ that is decreasing on $(M-\varepsilon, M + \varepsilon_0)$ that contains the range of $U_{t_1}-t_1$ and $U^*-t_1$. From the uniqueness result of \cite{BK}, $U_{t_1} = U^*$.

\

\noindent {\it Part 3:}  For $t_1 \in (M-\varepsilon', M)$, there exists $t_0 \ge t_1$ and a solution $U$ of \eqref{ut0} s.t. $U \ge u^{\sharp}$ in $B_{r(t_0)}:=\{u^{\sharp} > t_0 \}$, $U = u^{\sharp}$ on $\partial B_{r(t_0)}$ and $|\{ U > t_0\}| = |\{ u^{\sharp} > t_0 \} | $.

\noindent Using the properties for Schwarz symmetrization stated in Remark \ref{obs2}, the relations $\tilde{a}(t,{\rm z})\cdot {\rm z} \le a(t,{\rm z})\cdot {\rm z}$ and $\bar{F}(u^{\sharp})=F(u^{\sharp})$, and that $U_{t_1}$ minimizes $I_{t_1}$,
\begin{align*}
 \int_{\Omega_{t_1}} \frac{\nabla u \cdot a(u,\nabla u)}{\tilde{q}_0} - F (u) dx
 & \ge \int_{B_{r_1}} \frac{ \nabla u^{\sharp} \cdot a(u^{\sharp},\nabla u^{\sharp})}{\tilde{q}_0} - F (u^{\sharp}) dx \\
 & \ge \int_{B_{r_1}} \frac{ \nabla u^{\sharp} \cdot \tilde{a}(u^{\sharp},\nabla u^{\sharp})}{\tilde{q}_0} - \bar{F} (u^{\sharp}) dx \\
& \ge \int_{B_{r_1}} \frac{ \nabla U_{t_1} \cdot \tilde{a}(U_{t_1},\nabla U_{t_1})}{\tilde{q}_0} - \bar{F} (U_{t_1}) dx .\\
\end{align*}
Hence, from Lemma \ref{divergencia} and $\bar{F}(U_{t_1}) = F(U_{t_1})$, we have
$$ \int_{\Omega_{t_1}} \frac{(u-t_1)f(u)}{\tilde{q}_0} - F (u) dx \ge  \int_{B_{r_1}} \frac{ ( U_{t_1}-t_1) f( U_{t_1})}{\tilde{q}_0} - F (U_{t_1}) dx ,$$
that is equal to estimate \eqref{DesigualdadeChave}. Note also that
$$h_{t_1}(s)= \frac{(s-t_1)f(t)}{\tilde{q}_0} -F(s) $$
is decreasing in $(t_1,M+ \varepsilon_0)$ since $G_{\varepsilon}(s)$ is decreasing in this interval, where $\varepsilon=M-t_1 < \varepsilon_0$. Therefore, using that $U_{t_1}(B_{r_1}), u (B_{r_1}) \subset [t_1, M + \varepsilon_0)$
and an argument similar to the one that come after \eqref{DesigualdadeChave}, we have
$$\max u \le \max U_{t_1}.$$
If $u^{\sharp} \le U_{t_1}$ in $B_{r_1}$, Part 3 is proved taking $t_0=t_1$. Otherwise, there exist $t_2 \in (t_1,M)$ such that
$\mu_u(t_2) > \mu_{U_{t_1}}(t_2)$. Therefore $B'=\{u^{\sharp} > t_2\}$ and $B''=\{U_{t_1} > t_2 \}$ are concentric balls satisfying $|B'| > |B''|$.
Hence, from Theorem \ref{TeorSolucRadial} or \ref{TeorSolucRadialfnull}, there exists
some solution $U_{t_2}$ of \eqref{ut0} with $t_0$ replaced by $t_2$, such that $\{U_{t_2} > t_2 \} = B'$ and $U_{t_2} > U_{t_1}$ in $B''$. Since
$$ \max u^{\sharp} \le \max U_{t_1} < \max U_{t_2}, $$
it follows from the right continuity of $\mu_u$ and the continuity of $\mu_{U_{t_2}}$ that there exists $t_0 \ge t_2$,
such that $|\{U_{t_2} > t_0 \}| = |\{ u^{\sharp} > t_0 \}|$ and $U_{t_2} \ge u^{\sharp}$ in $\{ u^{\sharp} > t_0 \}$, proving this part.

\

\noindent {\it Part 4:} There exists a solution $U_{t_0}$ of \eqref{ut0} s.t. $U_{t_0} \ge u^{\sharp}$ in $B^*:=\{ u^{\sharp} \ge t_0 \}$, $U = u^{\sharp}$ on $\partial B^*$ and $|\{ U_{t_0} \ge t_0\}| = |\{ u^{\sharp} \ge t_0 \} | $.

\

\noindent Let $t_0$ and $U$ as in Part 3. If $|\{ u^{\sharp} \ge t_0\}|= \mu_u(t_0)$, then the theorem is proved with $B^*=\{ u^{\sharp} > t_0 \}$. Otherwise, applying Theorem \ref{TeorSolucRadial} or \ref{TeorSolucRadialfnull} for $B'=\{u^{\sharp} \ge t_0 \}$ and $B''=\{u^{\sharp} > t_0 \}$, there exists a solution $U_{t_0}$ of \eqref{ut0} s.t. $U_{t_0} > U$ in $B''$, proving the result with $B^*=B'$.
\end{proof}

Now we present a result that resemble a maximum principle for distribution function in the sense that the distribution $\mu_u$ of a solution cannot touch by below the distribution $\mu_U$ of
a radial solution if $\mu_u \le \mu_U$.

\begin{proposition}
Suppose that $a$, $\tilde{a}$ and $f$ satisfy $(H2)$-$(H4)$, where the constants and powers presented in $(H4)$ associated to $\tilde{a}$ are given by $(\tilde{C}_s, \tilde{C}_*, \tilde{C}^*)$ and $(\tilde{p}, \tilde{q}, \tilde{q}_0)$,
and that $\tilde{a}(t,{\rm z})\cdot {\rm z} \le a(t,{\rm z})\cdot {\rm z} $ for any $ {\rm z} \in \mathbb{R}^n$.
Assume also that $u \in W^{1,p}_0(\Omega)$ is a solution of $({\rm P}_{\Omega})$ and $U \in W^{1,p}(B)\cap C^1(B)$ is a radial solution of $(\tilde{{\rm P}}_{B})$ that not necessarily vanishes on $\partial B$.
If $u^{\sharp} \le U$ and
$u^{\sharp} \not\equiv U$, then there exists $t_1 \ge 0$ such that $u^{\sharp} < U$ in $\{ U > t_1\}$ and
$u^{\sharp} = U$ in $\{ U \le t_1\}$.

Moreover, assuming that $u^{\sharp} \le U$, if $f$ is strictly increasing and $u^{\sharp} \not\equiv U$, or $\Omega$ is not a ball and $a=a({\rm z})$ $(or \; \tilde{a}=\tilde{a}({\rm z}))$ satisfies hypotheses of Proposition \ref{continuidade}, then $u^{\sharp} < U$ in $B$.
\label{tipodeprincipiodomaximogeral}
\end{proposition}
\begin{proof}
Since $U \ge u^{\sharp}$ and $f$ is nondecreasing, we
have \begin{equation} \int_{\{U > t \}} f(U) \; dx \ge \int_{\{u^{\sharp} > t
\}} f(u^{\sharp}) \; dx = \int_{\{u > t \}} f(u) \; dx, \label{fintegralinequality}
\end{equation} for any $t \ge 0$.  Hence, applying Lemma \ref{divergencelemma} for $u$ and
$U$ and P\'olya-Szeg\"o principle, we get
\begin{equation}
\label{fintegralinequality2}
\int_{\{ U = t \}}
\frac{\tilde{a}(U, \nabla U) \cdot \nabla U}{|\nabla U|} \; dH^{n-1}
\ge \int_{\{ u^{\sharp} = t \}} \frac{a(u^{\sharp}, \nabla u^{\sharp}) \cdot \nabla
u^{\sharp}}{|\nabla u^{\sharp}|} \; dH^{n-1}
\end{equation} for almost every $t \ge \inf U$. Since $\tilde{a}(t,{\rm z})\cdot {\rm z} \le a(t,{\rm z})\cdot {\rm z} $, we have the same inequality with $a$ or $\tilde{a}$ appearing in both sides. Letting $r_1= (\mu_{u}(t)/\omega_n)^{1/n}$ and $r_2= (\mu_{U}(t)/\omega_n)^{1/n}$
we have some $t_0$ such that $r_1(t_0) < r_2(t_0)$ since $u^{\sharp} \not\equiv U$. Hence, Lemma \ref{comparacaoradial} implies that $u^{\sharp} < U$ on $\{U > t_0\}$.
Indeed, we can infer that the set of $t's$, for which $r_1(t)=r_2(t)$, is an interval that contains $0$. Denoting the supremum of this set by $t_1$, we have the first part of the result.

Now consider the case $f$ is strictly increasing and $t_1 >0$. Then we have a strict inequality in \eqref{fintegralinequality}
and, therefore, in \eqref{fintegralinequality2} for any $t \in [0,t_1]$, that contradicts $u^{\sharp} = U$ in  $\{ 0 \le U < t_1 \}$.

If  $a$ is as stated in Proposition \ref{continuidade}, it follows
from  $\tilde{a}({\rm z})\cdot {\rm z} \le a({\rm z})\cdot {\rm z}
$, \eqref{fintegralinequality}, Lemma \ref{divergencelemma},  and
P\'olya-Szeg\"o principle that
$$ \int_{  U < t } \nabla U \cdot  a(\nabla U )  dx  \ge \int_{  u < t } \nabla u \cdot a( \nabla u )  dx  \ge \int_{
u^{\sharp} < t } \nabla u^{\sharp} \cdot a( \nabla u^{\sharp} )  dx, $$
for $t < t_1$. Since $u^{\sharp}=U$ in $\{ U < t_1\}$, the three integrals are equals for $t < t_1$, and therefore, Proposition \ref{continuidade} implies that $u^{\sharp}$ is a translation of $u$ in $\{ u < t_1 \}$ and $\Omega$ is a ball,
that is an absurd.
Replacing $a$ by $\tilde{a}$, we see that the same conclusion holds if $\tilde{a}$ satisfies the hypotheses of that proposition.
\end{proof}

\begin{proof}{\it of Proposition \ref{primeiracomparacao}} Observe that $\tilde{a}$ and $f$ satisfy (H1)-(H5). Furthermore $\tilde{a}$ also satisfy $(H6)$, since $|\tilde{a}(t,{\rm z})| = \tilde{C}_s |{\rm z}|^{\tilde{q}_0 -1}$ for ${\rm z}$ small.
Then let $U_B$ be the solution stated in
Theorem \ref{TeorSolucRadial} or in Theorem \ref{TeorSolucRadialfnull} for $m=0$.  Consider the set
$$A= \{ t_0  : \exists {\rm \, a \, radial \, sol.} \;
U_{t_0} \; {\rm of} \; (\ref{ut0}) \, {\rm s.t.} \, U_{t_0} \ge
u^{\sharp} \, {\rm in} \, B^* \, {\rm and } \, |B^*| =
|\{ u^{\sharp} \ge t_0 \} | \}.$$ According to the previous lemma this set is
not empty. To prove the theorem, it suffices to show that $0
\in A$. For that  we prove the following assertions.

\

\noindent {\sl \underline{Assertion 1}:} For any positive $t_1 \in A$, there exists $t' \in A$ such that $t' < t_1$.
\\ From the definition of $A$, there exists a radial solution $U_{t_1}$ of
(\ref{ut0}) greater than or equal to $u^{\sharp}$ in $\{ u^{\sharp} \ge t_1 \}$. Since $U_{t_1}$ is radial, it can be extended as a positive radial solution of $-{\rm div} (\tilde{a}(V, \nabla V ))= f(V)$
in some ball that contains $\{ u^{\sharp} \ge t_1 \}$ or in $\mathbb{R}^n$. The maximal extension will be denoted by $U_{t_1}$. Consider
$$ D= \{ t \ge 0 : \; |\{U_{t_1} > t\}| = | \{ u \ge t \} | \; {\rm
and} \; |\{U_{t_1} > s\}| \ge | \Omega_{s}| \; {\rm for } \; s > t
\},$$ and let $t_2 =\inf D$. Observe that $t_1 \in D$ and so $t_2 \le t_1$.
If $t_2 < t_1$, then there exists $t_3 \in [t_2,t_1) \cap D$. Hence, in this case, our assertion is proved taking $t'=t_3$.
Consider now the case $t_2=t_1$. Thus $0 \not\in D$, since $0 <
t_1=t_2$. Therefore, there are two possibilities:

\

\noindent  1) $|\{
U_{t_1} > 0 \}| > |\Omega |$ and $|\{U_{t_1} > s\}| \ge |
\Omega_{s}|$ for $s > 0$; \\ 2) $|\{U_{t_1} > s_0\}| < |
\Omega_{s_0}|$ for some $s_0 \ge 0$.

\

\noindent Case 1): since $|\{U_{t_1} > s\}| \ge \mu_u(s) $ for $s > 0$, $U_{t_1} \ge u^{\sharp}$. Then, from the first part of Proposition \ref{tipodeprincipiodomaximogeral},
$U_{t_1} = u^{\sharp}$ in $\{ U_{t_1} < t_2 \}$, since $U_{t_1}= u^{\sharp}$ in $\{ U_{t_1}= t_2 \}$. However, this contradicts $|\{
U_{t_1} > 0 \}| > |\Omega |$ and, so this case is not possible.

\

\noindent Case 2): from the definition of $t_1$, it follows that $s_0 <
t_1$. Let $B'_{s_0}(0)$ be a ball such that $|B'_{s_0}| = |\{ u \ge s_0 \}|$.
Hence $B'=B'_{s_0}(0)$ and $B''= \{U_{t_1} > s_0\}$ satisfy $|B'| > |B''|$ and, from Theorem \ref{TeorSolucRadial} or \ref{TeorSolucRadialfnull}, there exists a
solution $U_{s_0}$ of $({\rm \tilde{P}}_{B'})$ with $U_{s_0}= s_0$ on $\partial B'$, such that $U_{s_0} > U_{t_1}$ in $B''$.
Then $U_{s_0} > U_{t_1} \ge u^{\sharp}$ in $\{U_{t_1} > t_1\}$ and, therefore,
$$\mu_{U_{s_0}}(t_1) = |\{U_{s_0} > t_1 \}| > |\{U_{t_1} > t_1\}|=|\{ u \ge t_1 \}|=\mu_u(t_1^-).$$
Since $\mu_{U_{s_0}}$ is continuous and $\mu_u(t_1^-) =  \lim_{t \to t_1^-} \mu_u(t)$, we have $\mu_{U_{s_0}}(t) > \mu_u(t)$ for $s_0 < t < t_1$, sufficiently close to $t_1$. Defining
$$ t' = \inf \{ t \ge s_0 \; : \;  \mu_{U_{s_0}}(t) > \mu_u(t) \},$$
it follows that $s_0 \le t' < t_1$ and $\mu_u(t') \le \mu_{U_{s_0}}(t') \le \mu_u(t'^-)$. Observe also that $U_{t_1} > u^{\sharp}$ in $\{ u^{\sharp} > t' \}$. Hence, this assertion is proved if $\mu_{U_{s_0}}(t')=\mu_u(t'^-)$.
If $\mu_{U_{s_0}}(t') < \mu_u(t'^-)$, applying Theorem \ref{TeorSolucRadial} or Theorem \ref{TeorSolucRadialfnull} for the balls $\{ U_{s_0} > t' \} \subsetneq \{ u^{\sharp} \ge t' \}$, we get a solution $U_{t'}$ s.t.
$U_{t'} > U_{s_0}$ in $\{ U_{s_0} > t' \}$ and $|\{ U_{t'} > t' \}| = |\{ u^{\sharp} \ge t'\}|$. Then $U_{t'} > u^{\sharp}$ in $\{ u^{\sharp} \ge t' \}$ and $U_{t'} = u^{\sharp} $ on $\partial \{ u^{\sharp} \ge t' \}$,
completing Assertion 1.

\

\noindent {\sl \underline{Assertion 2}:} If $t_1=\inf A$, then $t_1 \in A$.
\\ We can prove this using the same limit argument as in Lemma \ref{ContinuidadePsi}.

\noindent These assertions imply that $\inf A = 0$. Then there is a solution $U_0$ of $({\rm \tilde{P}}_{B})$
such that $U_0 \ge u^{\sharp}$. Since $U_B$ is maximal, it follows that $U_0 \le U_B$, proving the result.
\end{proof}

\begin{theorem} Let $\Omega \subset \mathbb{R}^n$ be a bounded
open set, $B$ be a ball centered at the origin with
$|B|=|\Omega|$, and suppose that $a$, $\tilde{a}$ and $f$ satisfy the hypotheses $(H1)$-$(H5)$, where the constants and powers associated to $a$ and $\tilde{a}$ may be different.
If $\tilde{a}(t,{\rm z})\cdot {\rm z} \le a(t,{\rm z})\cdot {\rm z} $ for any $ {\rm z} \in \mathbb{R}^n$, then there exists a radial solution $U_B \in W_0^{1,p}(B)$ of $({\rm \tilde{P}}_{B})$ such
that $$ U_B \ge u^{\sharp} \quad in \quad  B,$$ where $u^{\sharp}$
is the symmetrization of any solution $u$ of $($\ref{eq1}$)$.

Furthermore, if $\Omega$ is not a ball and $a=a({\rm z})$ $(or \; \tilde{a}=\tilde{a}({\rm z}))$ is as stated in Proposition \ref{continuidade}, then
$U_B > u^{\sharp}$.
\label{finaltheorem}
\end{theorem}

\begin{proof}
For $k \in \mathbb{N}$, let $a_k (t,{\rm z})=  b_k(t, |{\rm z}|) {\rm z}/ |{\rm z}|$ be a function satisfying (H3) s.t.
\\ $\bullet$ $|a_k| \le |\tilde{a}|$,
\\ $\bullet$ $a_k(t,{\rm z}) \cdot {\rm z} = C|{\rm z}|^{\tilde{q}_0}$ for some $C > 0$ and $|{\rm z}| \le 1/k$,
\\ $\bullet$ $a_k(t,{\rm z}) \cdot {\rm z} = \tilde{a}(t,{\rm z}) \cdot {\rm z}$ for $|{\rm z}| \ge 2/k$.
\\ To obtain such $a_k$, first observe that the convexity of $\tilde{a}(t,{\rm z}) \cdot {\rm z}$ in ${\rm z}$ and the relation $\tilde{a}(t,{\rm z}) \cdot {\rm z} \ge \tilde{C}_s |{\rm z}|^{\tilde{q}_0}$
imply that the derivative of $s \mapsto \tilde{a}(t,s{\rm w}) \cdot s{\rm w}$ is uniformly bounded from below by some $D_k>0$ for $t \in \mathbb{R}$, $|{\rm w}|=1$ and $s = 1/k$. From $\tilde{a} (t,{\rm z})=  \tilde{b}(t, |{\rm z}|) {\rm z}/ |{\rm z}|$, we get $\partial_s  [\tilde{b}(t,s) \, s ] \ge D_k$ for $s=1/k$ and $t \in \mathbb{R}$. Since $\tilde{a}(t,{\rm z}) \cdot {\rm z}$ in ${\rm z}$ is convex, $\partial_s  [\tilde{b}(t,s) \, s] $ is increasing in $s$ and,
then $\partial_s  \tilde{b}(t,s) \, s \ge D_k$ for $s=2/k$. Now define $b_k(t,s)$ in $\mathbb{R} \times [0,1/k]$ by $b_k(t,s)= C_k |s|^{\tilde{q}_0-1}$, where $C_k$ is such that $\partial_s[b_k(t,s) \, s]=D_k/2$ for $s=1/k$.
(Indeed we can chose $D_k = \tilde{C}_s (1/k)^{\tilde{q}_0-1}$ and $C=C_k= \tilde{C}_s/(2\tilde{q}_0)$.) Hence it is possible to extend $b_k$ to $\mathbb{R} \times [0,+\infty)$ in such a way that $\partial_s  [b_k(t,s) \, s] $ is strictly increasing in $s$, continuous and $b_k(t,s)=\tilde{b}(t,s)$ for $s\ge 2/k$. The function $a_k$ defined
from $b_k$ satisfies the required properties.

\

Since $a$, $a_k$ and $f$ satisfy the hypotheses of Proposition \ref{primeiracomparacao}, there exists some radial solution $U_k \in W_0^{1,p}(B)$ of $-{\rm div } \, a_k(V,\nabla V) = f(V)$ in $B$ that satisfies
$U_k \ge u^{\sharp}$, for any solution $u$ of $({\rm P}_{\Omega})$. Using \eqref{sup3}, it follows that the sequence $(U_k)$ is bounded in the $L^{\infty}$ norm and,
following the same argument as in Part 1 of Lemma \ref{SolucaoNoPico}, the derivative of $U_k$ is also uniformly bounded and equicontinuous. Hence, some subsequence converges to some function $U_B$ that is a
weak solution of $({\rm \tilde{P}}_B)$, by standart arguments. Moreover, $U_k \ge u^{\sharp}$ implies that $U_B \ge u^{\sharp}$, for any solution $u$ of $({\rm P}_{\Omega})$, completing the first part of the theorem.

\

Suppose now that $\Omega$ is not a ball and $u$ is a solution of $({\rm P}_{\Omega})$. From the first part, $U_B \ge u^{\sharp}$ and, therefore, applying Proposition \ref{tipodeprincipiodomaximogeral}, $U_B > u^{\sharp}$.
\end{proof}

\section{Existence and bound result}
\label{EBR}

First we apply the results of the previous section to prove that the symmetrization of solutions of \eqref{appequation} are bounded by a radial solution.
Notice that if $h$ is also bounded from above, the proof follows immediately from Theorem \ref{finaltheorem} applied to the equation $-{\rm div}( h(v) a(\nabla v) ) = f(v)$.
For $h$ just bounded from below by some positive constant, proceed as follows: let $m=\inf h$, $a_0(t,{\rm z})=m \, a({\rm z})$ and $a_1(t,{\rm z})=h(t) a({\rm z})$.
Since $a_0(t,{\rm z}) \cdot {\rm z} \le a_1(t,{\rm z}) \cdot {\rm z}$ and $a_0$ fulfill all necessary assumptions, Theorem \ref{finaltheorem} implies that there exists a solution $U_0$ for
$$ -m \, {\rm div} (a(\nabla V)) = f(V) \quad {\rm in } \quad B, $$
such that $U_0 \ge u^{\sharp}$, where $u$ is any solution of \eqref{appequation}. Let $M= \max U_0$, $h_1$ be a $C^1$ function such that $h_1(t)=h(t)$ for $t \le M$ and $h_1(t)=h(M+1)$ for $t \ge M+1$, and $a_2(t,{\rm z})=h_1(t)a({\rm z})$. Observe that $u$ is solution of $-{\rm div } (a_2(v,\nabla v)) = f(v)$ and $a_2$ satisfies
(H1)-(H5). Hence from Theorem \ref{finaltheorem}, there exists a radial solution $U_B$ of $-{\rm div } (a_2(V,\nabla V)) = f(V)$ in $B$ such that $U_B \ge u^{\sharp}$. Moreover $U_B \le U_0$, since $a_2 \ge a_0$.
Therefore $U_B$ is also a solution of$-{\rm div}( h(V) a(\nabla V) ) = f(V)$ completing the proof. This can be summarized in the next proposition.

\begin{proposition}
If $a_1=h a $ and $f=gh$ satisfy $(H1)$-$(H5)$, then there exists a radial function $U_B$, solution of \eqref{appequation} when the domain is $B$, such that $U_B \ge u^{\sharp}$, where $u^{\sharp}$ is the symmetrization of any solution of \eqref{appequation}. This is also true if $a_1$ does not satisfy the right inequality of \eqref{alqestimate}.
\end{proposition}

This result gives a priori estimate of a solution $u$, but does not prove its existence, except for the ball where we obtain the function $U_B$. We show now an existence result for a particular case, using this estimates.

\begin{theorem}
Let $a({\rm z})= {\rm z}|{\rm z}|^{p-2}$ and suppose that $a_1=h a $ and $f=gh$ satisfy $(H1)$-$(H5)$, with the possibility of not fulfillment of the right inequality of \eqref{alqestimate}.  Then there exists a solution $u$ to the problem \eqref{appequation}.
\end{theorem}

\begin{proof}
Let $M$, $h_1$ and $U_B$ be as defined before. Define the functional
$$ J(v) = \int_{\Omega} (h_1(v))^{\frac{p}{p-1}}\frac{|\nabla v|^p}{p} - \int_0^v f(s) (h_1(s))^{\frac{1}{p-1}} \, ds \; dx .$$
Since $h_1$ is bounded from above and from bellow by some positive constants, conditions (H4) and (H5) holds with $q=q_0=p$. Then
we can minimize $J$ in $W^{1,p}(\Omega)$ and obtain a solution $u$ to $-{\rm div} (h_1(v) v|v|^{p-2}) = f(v)$. From the previous result, we have that
$u$ is bounded by $U_B$ and, therefore, is a solution that we are looking for.
\end{proof}

\section{Estimates for Eigenfunctions}
\label{EP}

In the next result, the estimate \eqref{lpestimates1} and \eqref{eigenvalueestimate00} were established in \cite{C2} and \cite{C3} for $p=q=2$, with the best constant, and extended in \cite{AFT} for  $p=q>1$,
 when $\lambda$ is the first eigenvalue.

\begin{theorem}
Let $\Omega \subset \mathbb{R}^n$ be an open bounded set and $w$ be a solution of
\begin{equation} \left\{
\begin{array}{rcll}
- \Delta_p v & = & \lambda v|v|^{q-2}& \; \text{in } \; \; \Omega \\[5pt]
  v & = & 0 & \; \text{on } \; \; \partial \Omega \\ \end{array}
\right.   \label{autov1}
\end{equation}
where $1 < q \le p$ and $\lambda$ is either a real number if $q < p$ or any eigenvalue of $-\Delta_p$ with
trivial boundary data if $q=p$. Then
\begin{equation}
\label{lpestimates1}
(\max |w|)^{1+\frac{n(p-q)}{rp}}  \le \frac{2}{(\omega_n)^{1/r}} \left( \frac{2(p-1)}{p}\right)^{\frac{n(p-1)}{rp}} \left(\frac{\lambda}{n}\right)^{n/rp} \| w\|_r,
\end{equation}
for any $r >0$. Furthermore,
\begin{equation}
\label{lowerdistributionestimate}
 |\tilde{\Omega}_t| \ge \omega_n (\|w\|_{\infty}-t)^{\frac{n(p-1)}{p}} \left( \frac{p}{p-1}\right)^{\frac{n(p-1)}{p}} \left( \frac{n}{\lambda}\right)^{n/p} \|w\|_{\infty}^{\frac{n(1-q)}{p}},
\end{equation}
where $\tilde{\Omega}_t = \{ |w| > t \}$, $t \in [0,\max |w|]$.
\end{theorem}
\begin{proof}
 Let $M =\|w\|_{\infty}$, $\rho \ge 1$,  and $\Omega_2 = \{ x :  |w(x)|
> M/\rho\}$. Then
\begin{equation} \label{eigenvalueestimate1}
 \|w\|_r^r = \int_{\Omega} |w|^r \; dx \ge \int_{\Omega_2} |w|^r \;
dx \ge \left(\frac{M}{\rho}\right)^r |\Omega_2| \end{equation}  On
the other hand,
$$ -\Delta_p w = \lambda w|w|^{q-2} \le \lambda M^{q-1} $$ Hence, by the comparison principle of \cite{Da},
$|w| \le u$ in $\Omega_2$, where $u$ is solution of
$$ \left\{
\begin{array}{rcll}
- \Delta_p v & = & \lambda M^{q-1} & \; \text{in } \; \; \Omega_2 \\[5pt]
  v & = & \frac{M}{\rho} & \; \text{on } \; \; \partial \Omega_2 \\ \end{array}
\right.  $$

\noindent Let $U$ be the solution of
$$ \left\{
\begin{array}{rcll}
- \Delta_p V & = & \lambda M^{q-1} & \; \text{in } \; \; B \\[5pt]
  V & = & \frac{M}{\rho} & \; \text{on } \; \; \partial B \\ \end{array}
\right.  $$ where $B$ is a ball such that $|B| = |\Omega_2|$. From Theorem 1 of \cite{T2} or
Theorem \ref{teoremaP}, $u^{\sharp} \le
U$. Then $$ M= \max |w| \le \max u = \max u^{\sharp} \le \max U $$
We can compute $U$ explicitly:
$$U(x)= \frac{p-1}{p} \left(\frac{\lambda}{n}\right)^{\frac{1}{p-1}} M^{\frac{q-1}{p-1}}  \left( R^{\frac{p}{p-1}} - |x|^{\frac{p}{p-1}} \right) +
\frac{M}{\rho} ,$$ where $\omega_n R^n = |\Omega_2|=|B|$. Since $M
\le \max U= U(0)$, $$ M \le \frac{p-1}{p}
\left(\frac{\lambda}{n}\right)^{\frac{1}{p-1}} M^{\frac{q-1}{p-1}}
R^{\frac{p}{p-1}}  + \frac{M}{\rho}.$$ Hence,
$$ R \ge \left[ \frac{(\rho-1)p}{\rho(p-1)}\right]^{\frac{p-1}{p}} \left( \frac{n}{\lambda}\right)^{1/p} M^{\frac{p-q}{p}}.$$
Using this and $ R = \left(\frac{|\Omega_2|}{ \omega_n}
\right)^{1/n}$, we get
$$|\Omega_2| \ge \omega_n \left[ \frac{(\rho-1)p}{\rho(p-1)}\right]^{\frac{n(p-1)}{p}} \left( \frac{n}{\lambda}\right)^{n/p} M^{\frac{n(p-q)}{p}}.$$
From this, we get the estimate for $|\Omega_t|$ taking $t=M/\rho$.  Moreover applying this inequality with $\rho=2$ and using \eqref{eigenvalueestimate1}, it follows that
$$\|w\|_r^r \ge  \frac{1}{2^r} \omega_n \left( \frac{p}{2(p-1)}\right)^{\frac{n(p-1)}{p}} \left( \frac{n}{\lambda}\right)^{n/p}M^{\frac{n(p-q)}{p}+r}.$$
\end{proof}

\begin{remark}
\label{RM1secao6}
The estimates of this theorem still holds if $|\Delta_pw| \le |\lambda w|w|^{q-2}|$ or, equivalently, $-\Delta_p w = \lambda g(w)$, where $|g(w)|\le |w|^{q-1}$.
Hence, using the interpolation inequality,
$$ \|w\|_s \le \|w \|_{\infty}^{1-r/s} \|w\|_r^{r/s}, \quad \text{ for } 0 < r < s \le \infty $$
we get \eqref{eigenvalueestimate00} for solutions of $-\Delta_p w = \lambda g(w)$, where $|g(w)|\le |w|^{q-1}$, with the boundary condition $w=0$ on $\partial \Omega$.
Inequality \eqref{lpestimates1} is also true for solutions of ${\rm div}(a (x,Dw))\le |\lambda g(w)| $, provided $a:\Omega \times \mathbb{R}^n\to \mathbb{R}^n$ is such that
some comparison principle holds. For instance, consider the following hypotheses on $a$ given by  \cite{Da}:
$$ a  \in C(\bar{\Omega}\times\mathbb{R}^n; \mathbb{R}^n) \cap C^1(\bar{\Omega}\times(\mathbb{R}^n\backslash\{0\}); \mathbb{R}^n),$$
\begin{equation}
\label{hypothesisona}
\begin{array}{rl}
a(x,0)=0 & \quad {\rm for} \; x \in \Omega, \\[5pt]
\langle D_z a(x,{\rm z})\xi,\xi \rangle \ge (p-1)|{\rm z}|^{p-2}|\xi|^2 &  \quad {\rm for} \; (x, {\rm z}) \in \Omega \times \mathbb{R}^n \backslash\{ 0\},
\\[5pt]  |D_za(x,{\rm z})| \le C |z|^{p-2} & \quad {\rm for}  \; (x, {\rm z}) \in \Omega \times \mathbb{R}^n \backslash\{ 0\} \; {\rm, \;} C > 0.
\end{array}
\end{equation}
\end{remark}

\begin{theorem}
 Let $w$ be a bounded solution of
\begin{equation} \label{ProblemaDirichletNovo} \left\{ \begin{array}{rcll}
- {\rm div}(a (x, \nabla v)) & = & f(v)& \; \text{in } \; \; \Omega \\[5pt]
  v & = & 0 & \; \text{on } \; \; \partial \Omega, \\ \end{array}
\right.  \end{equation}
where $a$ satisfies \eqref{hypothesisona} and $f\in C^1(\mathbb{R})$ satisfies
$|f(t)| \le c|t|^{q-1} + d$, with $0 < q \le p$ and $c,d \ge 0. $ Then
$$ \|w\|_{\infty} \le \max\left\{ C_1 \|w\|_r^{\frac{rp}{n(p-q)+rp}},  C_2 \|w\|_{r}^{\frac{rp}{n(p-1) + rp}} \right\},  $$
where $C_1=C_1(n,p,q,r,\rho,c)$ and $C_2=C_2(n,p,r\rho,d)$ are positive constants.
\label{teoremaFinal}
\end{theorem}
\begin{proof}
We use the same ideas of the last theorem. By the comparison principle of \cite{Da}, $|w| \le u$, where $u$ solves  $- {\rm div}(a (x, \nabla v))=cM^{q-1} +d$ in $\Omega_2$
and  $u = M/\rho$ on $\partial \Omega_2$. Since the hypotheses on $a$ imply that $\langle a(x,{\rm z}),{\rm z} \rangle \ge |{\rm z}|^p$, using the same argument as in Remark \ref{remarktheorema1},
we have that $\max u \le \max U$, where $U$ is the solution of
$ -\Delta_p v = cM^{q-1} + d$ on $B$ and $v= M/\rho$ on $\partial B$. Notice that $U$ is given by
$$ U (x) = \frac{p-1}{p} \left(\frac{1}{n}\right)^{\frac{1}{p-1}} (cM^{q-1} + d)^{\frac{1}{p-1}}  \left( R^{\frac{p}{p-1}} - |x|^{\frac{p}{p-1}} \right) +
\frac{M}{\rho}. $$
Following the same computations as before, we conclude the proof where the constants are given by
$$C_1= (2c)^{\frac{n}{n(p-q)+rp}}K^{\frac{p}{n(p-q)+rp}} \quad , \quad  C_2=(2d)^{\frac{n}{n(p-1)+rp}}K^{\frac{p}{n(p-1)+rp}},$$ and $$K=\frac{1}{\omega_n} \left(1-\frac{1}{\rho}\right)^{-\frac{n(p-1)}{p}}\rho^r\left(\frac{p-1}{p}\right)^{\frac{n(p-1)}{p}} \left( \frac{1}{n}\right)^{\frac{n}{p}}. $$
\end{proof}

\noindent Using the interpolation inequality observed in Remark \ref{RM1secao6}, we can obtain estimates for $\|w\|_s$, where $s \in (r,\infty]$.

Now, we use this theorem to show that the $L^p$ norms of a solution goes to zero when its domain becomes ``far away'' from a ball with the same measure.
More precisely, when the first eigenvalue of a domain of a given measure is large, then its $L^p$ norms are small.

\begin{corollary}
Assuming the same hypotheses as in the previous theorem, if $p=q$ and $c < \lambda_p(\Omega)$, the first eigenvalue of $-\Delta_p$, then
$$ \|w\|_{\infty} \le \max \left\{ C_1 \left( \frac{d}{\lambda_p(\Omega) - c}\right)^{\frac{1}{p-1}}|\Omega|^{\frac{1}{p}} , C_2 \left( \frac{d}{\lambda_p(\Omega) - c}\right)^{\frac{\kappa_1}{p-1}}|\Omega|^{\frac{\kappa_1}{p}} \right\},$$
where $\kappa_1 = p^2/[n(p-1) + p^2]$. If $p > q$, then
$$ \|w\|_{\infty} \le \max \left\{ C_1 \tau^{\frac{rp}{n(p-q)+rp}},  C_2 \tau^{\frac{rp}{n(p-1) + rp}} \right\}  ,$$
where $\tau = |\Omega|^{1/p} \max \{ (2c/\lambda_p(\Omega))^{1/(p-q)}, (2d/\lambda_p(\Omega))^{1/(p-1)} \}$.\label{Linfinitestimatebyeigenvalue}
\end{corollary}
\begin{proof}
First note that the growth condition on $f$ and H\"older inequality imply
$$ \lambda_p(\Omega) \|w\|_p^p \le \|\nabla w\|_p^p \le \int_{\Omega} \nabla w \cdot a(\nabla w,x) dx \le c\|w\|_p^{q} |\Omega|^{\frac{p-q}{p}} + d \|w\|_p |\Omega|^{\frac{p-1}{p}}.$$
The proof for the case $p=q$ follows directly from this and Theorem \ref{teoremaFinal}. In the case $p > q$, we get from this inequality that $\|w\|_p \le \tau$. Hence we complete the proof applying Theorem \ref{teoremaFinal}.
\end{proof}

\begin{corollary}
Assume the same hypotheses about $a$ and $f$ as in the previous theorem. Suppose also that $a=a({\rm z})$, $f(t) > 0$ for $t> 0$ and $f(t)=0$ for $t \le 0$.
If $\lambda_p(\Omega)$ is sufficiently large, then any solution $u$ of \eqref{ProblemaDirichletNovo} in $\Omega$ satisfies $u^{\sharp} < U$, where $u^{\sharp}$ is the symmetrization of $u$
and $U$ is the maximal solution of \eqref{ProblemaDirichletNovo} in the ball $B$ with the same measure as $\Omega$.
\end{corollary}
The novelty in this corollary is that $f$ does not need to be monotone.
\begin{proof}
From Hopf lemma, $\partial_n U =c < 0$ on $\partial B$ and, therefore, there exists some ``paraboloid'' $$P(x)=\frac{p-1}{p} \left(\frac{1}{n}\right)^{\frac{1}{p-1}} C^{\frac{1}{p-1}}  \left( r^{\frac{p}{p-1}} - |x-x_0|^{\frac{p}{p-1}} \right),$$ where $x_0$ is the center of $B$ and $r$ is its radius, such that $0 < P < U$ in $B$. Observe that $-\Delta_p P = C$. Since $f$ is continuous and $f(0)=0$, let $M > 0$ be such that $f(t) < C$ for $t < M$. Corollary \ref{Linfinitestimatebyeigenvalue} implies that $\|u\|_{\infty} < M$, where $u$ is any solution of \eqref{ProblemaDirichletNovo}, if $\lambda_p(\Omega)$ is large enough. Then
$$ -{\rm div} \; a(\nabla u ) = f(u) \le C = -\Delta_p U,$$
and, from Theorem \ref{finaltheorem}, $u^{\sharp} \le P < U$ proving the result.
\end{proof}

\section{Appendix}
We show now Lemma \ref{moserLplemma} with the same arguments as in Theorem 3.11 of \cite{MZ}.

\begin{proof} {\it of Lemma \ref{moserLplemma}}
Let $K>0$, $\ell > 0$, $r \ge 1$, $\gamma=qr-q+1$,
$$v=P(u)=\min\{(u+K)^r, \ell^{r-1} (u+K) \}  $$ and $$\varphi = G(u)= \min\{(u+K)^{\gamma}, \ell^{\gamma-1} (u+K)\} - K^{\gamma} \in W^{1,q}_0(\Omega').$$
Then, using that $a(t,{\rm z})\cdot {\rm z} \ge C_*(|{\rm z}|^q - 1)$ for all ${\rm z} \in \mathbb{R}^n$ and $t \in \mathbb{R}$, we get
\begin{align*}
 \int_{\Omega'} |\nabla v |^q \, dx
&\le  \int_{\Omega'} | P'(u)|^q \, \left( \frac{\nabla u
\cdot a(u,\nabla u)}{C_*} + 1 \right) \; dx \\[5pt]
& \le  \int_{\Omega'} \frac{| P'(u)|^q}{G'(u)} \cdot \frac{\nabla \varphi
\cdot a(u,\nabla u)}{C_*} \; dx + \int_{\Omega'} | P'(u)|^q \; dx.
\end{align*}
Notice that $|P'(u)|^q / G'(u) = E$, where $E=1$
if $u + K> \ell$ and $E= r^q/\gamma$ if $u + K < \ell$.
Then, $E \le r^q$ and, using $\nabla \varphi
\cdot a(u,\nabla u) \ge 0$,
\begin{align} \int_{\Omega'} |\nabla v |^q \, dx &\le \frac{r^q}{C_{*}}
\int_{\Omega'} \nabla \varphi \cdot a(u,\nabla u) \; dx + \int_{\Omega'} | P'(u)|^q \; dx \nonumber \\[5pt] &= \frac{r^q}{C_{*}} \int_{\Omega'} f(u) G(u) \, dx + \int_{\Omega'} | P'(u)|^q \; dx.
\label{ineq1}
\end{align}
Observe now that, for $u + K < \ell$,
\begin{align*}
f(u)G(u) \le (\alpha u^{q-1} + \beta)\cdot (u+K)^{\gamma}
&\le \alpha (u+K)^{q-1+\gamma} + \beta\frac{(u+K)^{q-1+\gamma}}{K^{q-1}}\\
&\le v^q \left(\alpha + \frac{\beta}{K^{q-1}}\right).
\end{align*}
In a similar way, we can prove this inequality also for the case for $u + K \ge \ell$. Furthermore, for $u+K \le \ell$,
\begin{align*}
|P'(u)|^q = |r(u+K)^{r-1}|^q = r^q \frac{(u+K)^{rq}}{(u+K)^q} \le r^q \frac{v^q}{K^q},
\end{align*}
that is also true for $u + K > \ell$.
From these two inequalities and \eqref{ineq1}, we get
\begin{equation} \label{ineq2} \int_{\Omega'} |\nabla v |^q \, dx \le \left[ \frac{r^q}{C_*} \cdot
\left(\alpha + \frac{\beta}{K^{q-1}}\right) + \frac{r^q}{K^q} \right] \int_{\Omega'} v^q  dx.
\end{equation}

\

\noindent Now we study the cases $q > n$, $q < n$ and $q=n$
separately.

\

\noindent {\bf Case 1:} $q > n$.

\noindent Observe that for $r=1$, we get $v=u+K$. Using the
Morrey's inequality for $v-K \in W^{1,q}_0$,
$$  \| v-K \|_{C^{0,1-n/q}} \le \tilde{C_0}
\|v-K\|_{W^{1,q}}  \le C_0 (\|v\|_q + K |\Omega'|^{1/q} +
\|Dv \|_q),$$ where $C_0=C_0(n,q)$. From this one and
\eqref{ineq2}, we get
$$ \sup u = \sup v-K \le \left[ C_0 + \!  \left[ \frac{1}{C_*} \cdot
\left(\alpha + \frac{\beta}{K^{q-1}}\right) + \frac{1}{K^q} \right]^{1/q} \right] \! \|v\|_q + C_0
K|\Omega'|^{1/q}. $$ Since $\|v\|_q= \|u+K\|_q \le \|u \|_q + K
|\Omega'|^{1/q}$, we get
$\sup u \le D_1 \|u\|_q + D_2 K |\Omega'|^{1/q}$,
where $$D_1= C_0 +   \left[ \frac{1}{C_*} \cdot
\left(\alpha + \frac{\beta}{K^{q-1}}\right) + \frac{1}{K^q} \right]^{1/q}  \quad \text{ and } \quad
D_2=D_1+C_0.$$

\

\noindent {\bf Case 2:} $q < n$.

\noindent Since $v - K^r \in W_0^{1,q}(\Omega')$, the Sobolev
inequality implies
\begin{equation} \label{eq2a}
 \| v-K^r \|_{q*} \le C_0\| \nabla v\|_q ,
\end{equation} where $q^* = nq/(n-q)$ and $C_0=\frac{q(n-1)}{n-q}$. Using
this and \eqref{ineq2}, we get
 \begin{equation} \label{eq2b}
\| v-K^r \|_{q*} \le C_0 \left[ \frac{r^q}{C_*} \cdot
\left(\alpha + \frac{\beta}{K^{q-1}}\right) + \frac{r^q}{K^q} \right]^{1/q} \|v\|_q
 .\end{equation} Hence, naming $\chi = n/(n-q)$, it follows that
$$ \| v \|_{\chi q} \le  C_0 \left[ \frac{r^q}{C_*} \cdot
\left(\alpha + \frac{\beta}{K^{q-1}}\right) + \frac{r^q}{K^q} \right]^{1/q} \|v\|_q  + K^r |\Omega'|^{1/\chi q} ,$$ that is,
$\| v \|_{\chi q} \le D_1 \|v\|_q + D_2 ,$ where
$$ D_1=C_0 \left[ \frac{r^q}{C_*} \cdot
\left(\alpha + \frac{\beta}{K^{q-1}}\right) + \frac{r^q}{K^q} \right]^{1/q}
\quad \text{ and } \quad D_2 =K^r |\Omega'|^{1/\chi q}.$$ Since $v$
depends on $r$ and $\ell$, we name it by $v_{r,\ell}$. In the same
way, $D_1=D_1(r)$ and $D_2=D_2(r)$. Hence, the last
inequality can be rewritten as
\begin{equation} \label{eq3} \| v_{r,l} \|_{\chi q} \; \le \; D_1(r) \, \|v_{r,l}\|_q + D_2(r) .\end{equation}
Taking $r=1$, we have $v=u+K$ and, then
$$\| u+ K \|_{\chi q} \; \le \; D_1(1) \, \|u+K\|_q + D_2(1)
.$$

\noindent Hence $u+K \in L^{\chi q}$ and, therefore, $(u+K)^{\chi}
\in L^{q}$. Taking $r=\chi$, we have $|v_{\chi,\ell}| \le
(u+K)^{\chi}$ for any $\ell$. Thus $\|v_{\chi,\ell}\|_{q} \le
\|(u+k)^{\chi}\|_{q}$ and, from \eqref{eq3},
$$ \|v_{\chi,\ell}\|_{\chi q} \le D_1(\chi) \|(u+K)^{\chi}\|_q + D_2(\chi). $$
Using that $v_{\chi,\ell} \uparrow (u+K)^{\chi}$ as $\ell \to
\infty$, we get
$$ \|(u+K)^{\chi}\|_{\chi q} \le D_1(\chi) \|(u+K)^{\chi}\|_q + D_2(\chi). $$
Therefore, $u+K \in L^{\chi^2 q}$. More generally, if we take $r = \chi^n$, it
follows in a similar way that
$$ \|(u+K)^{\chi^n}\|_{\chi q} \le D_1(\chi^n) \|(u+K)^{\chi^n}\|_q + D_2(\chi^n) $$
and $u+K \in L^{\chi^{n+1} q}$. Thus $u+K$ is an $L^r$ function
for any $r \ge 1$. Hence, making $\ell \to \infty$ in \eqref{eq3},
we get $$ \| u+K \|_{r\chi q}^r \; \le \; D_1(r) \,
\|u+K\|_{rq}^r + D_2(r). $$ Observe now that
$ D_1(r) = r H$,
where $$H = C_0 \left[ \frac{1}{C_*} \cdot
\left(\alpha + \frac{\beta}{K^{q-1}}\right) + \frac{1}{K^q} \right]^{1/q}. $$
Furthermore, $$ D_2(r) = K^r | \Omega'|^{1/\chi q}
\le \frac{\| u+K \|_{rq}^r}{ |\Omega'|^{1/q}} | \Omega'|^{1/\chi q}
 \le r \| u+K \|_{rq}^r |\Omega'|^{(1/\chi -1)1/q}.
$$
Therefore, the last three relations imply
\begin{equation} \label{ineq4}
\| u+K \|_{r\chi q}  \le r^{1/r} H_0^{1/r} \|u+K\|_{rq},
\end{equation}
for $r \ge 1$ and $\chi = n/(n-q)$, where $H_0 = H + |\Omega'|^{(1/\chi -1)1/q}$. Taking $r=\chi^{m}$ in \eqref{ineq4}, we have
$$\| u+K \|_{\chi^{m+1} q } \; \le \; \chi^{m/\chi^{m} } H_0^{1/\chi^{m} } \|u+K\|_{\chi^{m} q} \quad \text{ for } m \in \mathbb{N} \cup \{0\}.  $$
Hence, defining $A_m =  \sum_{j=0}^{m} j /\chi^j $ and
$B_m= \sum_{j=0}^{m} 1 /\chi^j$, it follows that $$\| u+K \|_{\chi^{m+1} q } \; \le \; \chi^{A_m } H_0^{B_m } \|u+K\|_{q} \quad \text{ for }  m \in \mathbb{N} \cup \{0\}. $$
Since $A_m$ and $B_m$ are
convergent series,
$$ \sup (u + K) \le \chi^A H_0^B \|u+K\|_{q}, $$
where $A=\lim_{m \to \infty}A_m $ and $B = \lim_{m \to \infty}B_m
= \frac{\chi}{\chi-1} $. Then \begin{align*} \sup u &\le
D (H^B + |\Omega'|^{B (1/\chi q - 1/q)})( \|u\|_{q} + \|K\|_q ),
\end{align*}
for $D= \chi^A 2^{B}$. Observe that $B (\frac{1}{\chi q} -
\frac{1}{q})=-\frac{1}{q}$. Therefore
$$ \sup u \le D(H^B + |\Omega'|^{-1/q} ) ( \|u\|_{q} + K |\Omega'|^{1/q}).$$
Notice that
$$ H \le C_0 2^{2/q} \left( \frac{\alpha}{C_*} + \frac{\beta}{C_*} + 1 \right)\left( 1 + \frac{1}{K} \right).  $$
Then, taking $K = |\Omega'|^{1/n}$, it follows that $$H^B \le C_1 (|\Omega'|^{1/n} + 1)^{B} |\Omega'|^{-1/q},$$ where
$C_1 = [C_0 2^{2/q} \left( \alpha/C_* + \beta/C_* + 1 \right)]^{B} $.
Hence
\begin{align*} \sup u &\le 2 D
 C_1 \left(|\Omega'|^{1/n} + 1\right)^{B} |\Omega'|^{-1/q} \left( \|u\|_{q} + K |\Omega'|^{1/q}\right) \\[5pt] &\le
C (|\Omega'|^{1/n} + 1)^{B}\left( |\Omega'|^{-1/q}  \|u\|_{q} + |\Omega'|^{1/n}\right),
\end{align*}
proving the result.

\

\noindent {\bf Case 3:} $q = n$: Taking $\tilde{q} < q=n$, we get
the same estimate as in \eqref{eq2a} with $\tilde{q}^*$ instead of
$q^*$. Hence
$$
 \| v-K^r \|_{\tilde{q}^*} \le C_0\| \nabla v\|_{\tilde{q}} ,
$$ where $\tilde{q}^* = n\tilde{q}/(n-\tilde{q})$ and
$C_0=\frac{\tilde{q}(n-1)}{n-\tilde{q}}$. Therefore, from H\"older
inequality,
$$ \| v-K^r \|_{\tilde{q}^*} \le C_0 \| \nabla v\|_{q}
|\Omega'|^{(q-\tilde{q})/q\tilde{q}}.$$ For $\tilde{q} > n/2$ we
get $\tilde{q}/(n-\tilde{q}) > 1$ and, then, $\tilde{q}^* > n=q$.
In this case, $$ \| v-K^r \|_{\chi q} \le C_0 \| \nabla v\|_{q}
|\Omega'|^{(q-\tilde{q})/q\tilde{q}},$$ where $\chi = \tilde{q}^*/q
> 1$. Using this and \eqref{ineq2}, it follows that
$$\| v-K^r \|_{\chi q} \le  C_0 \left[ \frac{r^q}{C_*} \cdot
\left(\alpha + \frac{\beta}{K^{q-1}}\right) + \frac{r^q}{K^q} \right]^{1/q} \|v\|_q
|\Omega'|^{(q-\tilde{q})/q\tilde{q}}.$$ This estimate is basically the same
as in \eqref{eq2b}. Hence, taking $K=|\Omega|^{1/n}$ and following the same argument as before we get the result.
\end{proof}

\

\noindent {\bf Acknowledgment}

\

\noindent This collaborative research is co-sponsored by the J.
Tinsley Oden Faculty Fellowship Program in the Institute for
Computational Engineering and Sciences at The University of Texas
at Austin.

\end{document}